\documentclass[hidelinks,onefignum,onetabnum]{siamart220329}

\usepackage{lipsum}
\usepackage{amsfonts}
\usepackage{graphicx}
\usepackage{epstopdf}
\usepackage{enumerate}
\usepackage{enumitem}
\usepackage{comment}
\usepackage[caption=false]{subfig}
\ifpdf
  \DeclareGraphicsExtensions{.eps,.pdf,.png,.jpg}
\else
  \DeclareGraphicsExtensions{.eps}
\fi

\newsiamremark{remark}{Remark}
\newsiamremark{hypothesis}{Hypothesis}
\crefname{hypothesis}{Hypothesis}{Hypotheses}
\newsiamthm{claim}{Claim}

\DeclareMathOperator*{\argmin}{arg\,min}

\headers{Convergence Analysis of AAP Method}{X. Feng, M.~P. Laiu, and T. Strohmer}

\title{Convergence analysis of the Alternating Anderson-Picard method for nonlinear fixed-point problems\thanks{Submitted to the editors DATE: July 14, 2024.
\funding{
This work was supported, in part, by the Office of Advanced Scientific Computing Research and performed at the Oak Ridge National Laboratory, which is managed by UT-Battelle, LLC for the US Department of Energy under Contract No. DE-AC05-00OR22725. T.S.\ and X.F.\ acknowledge support from NSF DMS-2208356, NIH R01HL16351. T.S., X.F.\ and P.L. acknowledge support from DE-AC05-00OR22725.\\
This manuscript has been authored, in part, by UT-Battelle, LLC under Contract No. DE-AC05-00OR22725 with the U.S. Department of Energy. The United States Government retains and the publisher, by accepting the article for publication, acknowledges that the United States Government retains a non-exclusive, paid-up, irrevocable, world-wide license to publish or reproduce the published form of this manuscript, or allow others to do so, for United States Government purposes. The Department of Energy will provide public access to these results of federally sponsored research in accordance with the DOE Public Access Plan(\url{http://energy.gov/downloads/doe-public-access-plan}).}}}

\author{Xue Feng\thanks{University of California, Davis, Davis, CA
  (\email{xffeng@ucdavis.edu},  \email{strohmer@math.ucdavis.edu}).}
\and M.~Paul Laiu\thanks{Oak Ridge National Laboratory, Oak Ridge, TN 
  (\email{laiump@ornl.gov}).}
\and Thomas Strohmer\footnotemark[2]}

\usepackage{amsopn}
\DeclareMathOperator{\diag}{diag}

\newtheorem{assumption}{Assumption}[section]

\usepackage{algorithm}
\usepackage{algpseudocode}

\newcommand{\bbR}{\mathbb{R}}
\newcommand{\bbB}{\mathbb{B}_t}

\newcommand{\bfalpha}{\boldsymbol{\alpha}}

\newcommand{\contraction}{\kappa}
\newcommand{\E}{m}

\newcommand{\spans}{\mathrm{span}}
\newcommand{\cond}{\mathrm{cond}}
\newcommand{\proj}{\mathcal{P}}
\newcommand{\damping}{\beta_t}
\renewcommand{\d}{\mathrm{d}}
\newcommand{\x}{\boldsymbol{x}}
\newcommand{\f}{\boldsymbol{f}}

\newcommand{\s}{\boldsymbol{s}}
\newcommand{\y}{\boldsymbol{y}}
\newcommand{\p}{\boldsymbol{p}}
\newcommand{\z}{\boldsymbol{z}}
\renewcommand{\r}{\boldsymbol{r}}
\renewcommand{\b}{\boldsymbol{b}}
\renewcommand{\v}{\boldsymbol{v}}
\newcommand{\Et}{\boldsymbol{E}_t}
\newcommand{\Jt}{\boldsymbol{J}_t}
\newcommand{\Bt}{\boldsymbol{B}_t}
\newcommand{\Gt}{\boldsymbol{G}_t}
\newcommand{\Ft}{\boldsymbol{F}_t}
\newcommand{\St}{\boldsymbol{S}_t}
\newcommand{\Yt}{\boldsymbol{Y}_t}
\newcommand{\Ht}{\boldsymbol{H}_t}
\newcommand{\bfA}{\boldsymbol{A}}
\newcommand{\bfK}{\boldsymbol{K}}
\newcommand{\bfI}{\boldsymbol{I}}
\newcommand{\bfB}{\boldsymbol{B}}
\newcommand{\bfS}{\boldsymbol{S}}
\newcommand{\bfY}{\boldsymbol{Y}}
\newcommand{\bfD}{\boldsymbol{D}}
\newcommand{\bfQ}{\boldsymbol{Q}}
\newcommand{\bfV}{\boldsymbol{V}}
\newcommand{\bfW}{\boldsymbol{W}}
\newcommand{\bfH}{\boldsymbol{H}}
\newcommand{\Deltat}{\boldsymbol{\Delta}_t}
\newcommand{\bfDelta}{\boldsymbol{\Delta}}
\newcommand{\bfepsilon}{\boldsymbol{\epsilon}}
\newcommand{\bfLambda}{\boldsymbol{\Lambda}}

\usepackage{xcolor}

\newcommand{\edit}[1]{{\color{black} #1}}

\ifpdf
\hypersetup{
  pdftitle={Convergence analysis of the Alternating Anderson-Picard method for nonlinear fixed-point problems},
  pdfauthor={X. Feng, M.~P. Laiu, and T. Strohmer}
}
\fi

\begin{document}

\maketitle

\begin{abstract}
Anderson Acceleration (AA) has been widely used to solve nonlinear fixed-point problems due to its rapid convergence. This work focuses on a variant of AA in which multiple Picard iterations are performed between each AA step, referred to as the Alternating Anderson-Picard (AAP) method. Despite introducing more ``slow'' Picard iterations, this method has been shown to be efficient and even more robust in both linear and nonlinear cases.
However, there is a lack of theoretical analysis for AAP in the nonlinear case.
In this paper, we address this gap by establishing the equivalence between AAP and a multisecant-GMRES method that uses GMRES to solve a multisecant linear system at each iteration. 
\edit{
From this perspective, we show that AAP ``converges'' to the Newton-GMRES method. 
Specifically, as the residual approaches zero, the multisecant matrix, the approximate Jacobian inverse, the search direction, and the optimization gain of AAP converge to their counterparts in the Newton-GMRES method.
These connections provide insights for analyzing the asymptotic convergence properties of AAP.
Consequently, we show that AAP is locally $q$-linear convergent and provide an upper bound for the convergence factor of AAP. 
To validate the theoretical results, numerical examples are provided.}

\end{abstract}



\section{Introduction}
Fixed-point iterations are one of the cornerstones in 
scientific computing, formulated by
\begin{equation}
\label{fixedpoint}
    \x = g(\x),
\end{equation}
where \( \x \in \mathbb{R}^d \) and \( g: \mathbb{R}^d \rightarrow \mathbb{R}^d \) is assumed to be a continuously differentiable operator in this work. 
The standard method for solving \eqref{fixedpoint} is the fixed-point iteration, also known as the Picard iteration,
noted for its simplicity but also for its potentially slow convergence.
This paper focuses on an acceleration method for nonlinear fixed-point iterations.

Anderson acceleration (AA), initially introduced in~\cite{anderson1965iterative} to address partial differential equations, is one of the most popular acceleration schemes.
AA is a scheme mixing history points where the mixing coefficients are computed by solving a least-squares (LS) problem.
It can be regarded as a multisecant quasi-Newton method whose approximate Jacobian inverse satisfies a multisecant equation~\cite{fang2009two}.
AA is known for its simplicity and significantly improved convergence performance. 
It has wide applications in fields such as electronic structure calculations~\cite{banerjee2016periodic,fang2009two}, fluid dynamics~\cite{Lott2012}, geometrical optimization~\cite{Peng2018}, and more recently machine learning~\cite{geist2018anderson,wang2021asymptotic}.

The classical AA method takes an AA step in each iteration.
There are different variants of AA such as restarted AA~\cite{ouyang2020anderson},
type-I AA~\cite{fang2009two, zhang2020globally}, and
EDIIS~\cite{chen2019convergence}. 
Among these alternatives, a simple method that applies AA at periodic intervals rather than every iteration draws our attention.
Contrary to the belief that a reduced frequency of AA would degrade performance, this technique has been proven to enhance both the efficiency and robustness of AA. This method was first introduced to address large-scale linear systems within the classical Jacobi fixed-point iteration 
framework for electronic structure calculations, termed the Alternating Anderson-Jacobi (AAJ) method in~\cite{pratapa2016anderson}. 
It has been shown to significantly outperform both the GMRES and Anderson-accelerated Jacobi methods. 
Subsequently, AAJ was generalized to include preconditioning, known as the Alternating Anderson-Richardson (AAR) method. It 
outperforms classical preconditioned Krylov solvers in terms of efficiency and scalability~\cite{suryanarayana2019alternating}. 
In nonlinear scenarios,~\cite{banerjee2016periodic} uses this strategy to accelerate self-consistent field iterations,
and \cite{han2023riemannian} applied a similar strategy on Riemannian optimization.
On the theoretical side,~\cite{lupo2019convergence} showed the equivalence of AAR to GMRES. 
However, no theoretical convergence analysis has been provided  for  this alternating approach when solving nonlinear fixed-point problems, which is the primary focus of this work.

We consider a specific variant of this periodically alternating method, referred to as AAP, which takes $m$ Picard iterations between each AA step and each AA step  mixes  $m+1$ history points.
This way, the algorithm is restarted after each AA step.
We refer to the method AAP($m$) in this paper when the number of Picard iterations $m$ needs to be specified.
A detailed description of AAP is given in \Cref{algo:AAP}.
This method takes advantage of the efficiency of AA to improve the convergence, while maintaining the simplicity of the Picard iterations. 

Solving \eqref{fixedpoint} is equivalent to finding the root of the residual function:
\begin{equation}
    0 = f(\x) := g(\x) - \x .
\end{equation}
Newton-type methods are fundamental in this context due to their characteristic fast convergence,
achieved by utilizing Jacobian information~\cite{nocedal1999numerical}.
One such method, Newton-GMRES \cite{choquet1993some,an2007globally}, iteratively updates the solution by
solving the root of the linear approximation of $f$ using GMRES ~\cite{knoll2004jacobian}.
\edit{In this work, we investigate the connection between AAP and other solvers for nonlinear systems and use the connections to prove the local $q$-linear convergence property of AAP.
Specifically, we first establish the equivalence between AAP and a multisecant-GMRES method, which iteratively uses GMRES to find the root of a multisecant linear approximation of $f$.
Based on the equivalence between AAP and multisecant-GMRES, we connect AAP to the Newton-GMRES method by showing that the multisecant matrix, the approximate Jacobian inverse, the search direction and the optimization gain of AAP (see Section~\ref{sec:review} and Equation \eqref{equ:thetaaap} for relevant discussions and definition) all converge to their counterparts in the Newton-GMRES method as the residual approaches zero.
By expressing AAP as a multisecant-GMRES method, we can bound the residual norm at the current AAP iteration by a linear and a quadratic term in terms of the residual norm at the previous iteration. The convergence of AAP to Newton-GMRES then allows us to control the quadratic term when the residual is small, which then leads to the local $q$-linear convergence property of AAP. Further, the convergence of the optimization gain in AAP to the Newton-GMRES counterpart gives an upper bound on the asymptotic $q$-linear convergence factor of AAP.
}

This paper is organized as follows. Section \ref{sec:algoandbackground} details the AAP algorithm and related work. Section \ref{sec:aapequivalence} demonstrates that, at each global iteration \( t \), AAP(\( m \)) effectively solves a multisecant linear system 
using the GMRES(\( m \)) method, followed by one additional linearized Picard iteration.
\edit{
Based on the results in Section~\ref{sec:aapequivalence},
Section~\ref{sec:relation} establishes the connection between AAP and the Newton-GMRES method.
In Section~\ref{sec:conv}, by using the relation between AAP, multisecant-GMRES, and Newton-GMRES established in Sections~\ref{sec:aapequivalence} and \ref{sec:relation}, the local $q$-linear convergence property of AAP is proved and an upper bound on the asymptotic $q$-linear convergence factor is provided.}
Numerical examples are provided in Section~\ref{sec:numerical-results}.
\Cref{sec:boundsing} discusses an assumption related to the convergence of AAP.

\label{sec:notation}
\subsection{Notation}
The vector norm is denoted by \(\|\cdot\|\) and refers to the Euclidean norm (2-norm). The matrix norm is also denoted by \(\|\cdot\|\) and refers to the induced 2-norm, which is the largest singular value. 
A sequence of matrices \(\bfA_k \rightarrow \bfA\) means \(\|\bfA_k - \bfA\| \rightarrow 0\).
In this paper,
a sequence of matrices \(\bfA_k \rightarrow \widetilde{\bfA}_k\), where $\widetilde{\bfA}_k$ is another sequence of matrices,  means \(\|\bfA_k - \widetilde{\bfA}_k\| \rightarrow 0\);
we refer to this as the convergence of \(\bfA_k \) to \( \widetilde{\bfA}_k\).
The Frobenius norm of the matrix is denoted as $\|\cdot\|_F$.
The identity matrix is denoted by $\bfI$.
For better presentation in this paper, 
$\bfA^{\ell}$ represents a superscript on matrix $\bfA$;
$\bfA^{(\ell)}$ refers to the matrix $\bfA$ raised to the power $\ell$. 
$\bfA^{-1}$ denotes the inverse of matrix $\bfA$.
$\bfA^{\dagger}$  denotes the left pseudoinverse of matrix $\bfA$.
The $i$th largest singular value of $\bfA$ is denoted as $\sigma_{i}(\bfA)$ , and $\sigma_{\text{min}}(\bfA)$ is the smallest singular value of $\bfA$.
The condition number of $\bfA$ is denoted as $\cond(\bfA)$.
The subspace spanned by the columns of matrix $\bfA$ is denoted as $\spans(\bfA)$.

For a given matrix \(\bfA\) and vector \(\b\), the \(n\)-th Krylov subspace is defined as:
\[
\mathcal{K}_n(\bfA, \b) = \mathrm{span}\{\b, \bfA\b, \bfA^2\b, \ldots, \bfA^{n-1}\b\}.
\]
Krylov subspaces are  shift invariant, i.e., for any $\nu \in\bbR$,
$\mathcal{K}_n(\bfA-\nu \bfI, \b)  =
\mathcal{K}_n(\bfA, \b) $.
The Generalized Minimal Residual Method (GMRES) 
is one of the most popular Krylov subspace methods
and solves the linear systems $\bfA\x=\b$ by minimizing the residual within a Krylov subspace at iteration $n$:
\[
\min_{\x \in \mathcal{K}_n(\bfA, \b)} \|\bfA\x - \b\|.
\]

The following assumptions will be used in this paper:

\begin{assumption}
\label{assumption:gaap}
The following properties are assumed for the operator $g$:
    \begin{enumerate}[label=(\alph*).]
        \item $
        \|g(\x) - g(\y)\| \leq \contraction \|\x  - \y\|, \quad \forall\,\,\x, \y \in \mathbb{R}^d
        $ with $\contraction\leq 1 $.%
        \footnote{\edit{We note that the results in Sections~\ref{sec:aapequivalence} and \ref{sec:relation} still hold when $g$ is expansive, i.e., $\contraction>1$, but with a larger, $\contraction$-dependent constant in some of the estimates.}}
        \item The Jacobian of $g$, denoted as $g'$, is Lipschitz continuous with constant $\gamma>0$,
        \[
        \|g'(\x) - g'(\y)\| \leq \gamma \|\x  - \y\|.
        \]
    \end{enumerate}
\end{assumption}

\section{Algorithm and related work}
\label{sec:algoandbackground}

This section provides the details of the AAP algorithm, along with the necessary preliminary concepts and assumptions. In addition, existing work on the convergence analysis of the classical AA is discussed. 

\subsection{The AAP algorithm}

At global iteration $t$, AAP first takes $m$ Picard iterations from current iterate $\x_t$ to generate new points $\{\x_{t}^1, \cdots, \x_{t}^m\}$ and compute their residuals.
AAP next computes the mixing coefficients by solving a constrained LS problem based on the residuals of the history points $\{\x_t, \x_{t}^1, \cdots, \x_{t}^m\}$, and then updates the iterate to $\x_{t+1}$ by mixing the fixed-point values of these points.
See \Cref{algo:AAP} for details. 
We denote the mixing coefficients at global iteration $t$ as $\bfalpha_t = \{ \alpha_t^0,\alpha_t^1,..., \alpha_t^m \}\in \bbR^{m+1}$, and denote the residual $f(\x_t)$ as either $\f_t$ or $\f_t^0$, depending on the context.
When relevant, we refer to this algorithm as AAP($m$), in which $m+1$ Picard iterations are performed in each global iteration.

\begin{algorithm}
\caption{Alternating Anderson-Picard (AAP) Method}
\label{algo:AAP}
\begin{algorithmic}[1] 
\State \textbf{Initialization:} $\x_0 \in \mathbb{R}^d$, \textit{number of Picard steps} $m\leq d$, \textit{damping ratio} $\damping\geq 0$
\For{\textit{global iteration} $t=0$ \textbf{to} $T$}
    \State \texttt{/* Step 1: takes Picard iteration to generate $m$ points and compute their residuals */}
    \State $\x_t^{0} \leftarrow \x_t$
    \For{$\ell=1$ \textbf{to} $m$}
        \State $\x_{t}^{\ell} \leftarrow g(\x_{t}^{\ell-1})$
        \State $\f_t^{\ell-1} \leftarrow \x_{t}^{\ell}-\x_{t}^{\ell-1}$ 
    \EndFor
    \State
    $\f_t^{m} \leftarrow g(\x_{t}^{m}) - \x_{t}^{m}
    $
    \State \texttt{/* Step 2: Take one Anderson Acceleration step */}
    \State Solve 
    \begin{equation}
    \label{equ:constrainedLS}
        \bfalpha_t 
        \leftarrow
        \argmin_{\alpha \in \mathbb{R}^{m+1}} \big\| \sum_{\ell=0}^{m} \alpha^{\ell} \f_t^{\ell} \big\|^2
        \quad \text{s.t.} \quad 
        \sum_{\ell=0}^{m} \alpha^{\ell} = 1 
    \end{equation}
    \State Update
    \begin{equation}\label{equ:AAPupdate}
        \x_{t+1} \leftarrow (1-\damping) \sum_{\ell=0}^{m} \alpha_t^{\ell} \x_t^{\ell} + \damping \sum_{\ell=0}^{m} \alpha_t^{\ell} g(\x_t^{\ell})
    \end{equation}
\EndFor
\end{algorithmic}
\end{algorithm}

Here, the damping ratio $\beta_t$ balances the weighted average of history points and the weighted average of history residuals, i.e.,
    \begin{equation}
    \label{equ:dampinaap}
            \begin{aligned}
        \x_{t+1} = (1-\damping) \sum_{\ell=0}^{m} \alpha_t^{\ell} \x_t^{\ell} + \damping \sum_{\ell=0}^{m} \alpha_t^{\ell} g(\x_t^{\ell}) =   \sum_{\ell=0}^{m} \alpha_t^{\ell}  \x_t^{\ell} + \damping \sum_{\ell=0}^{m} \alpha_t^{\ell} \f_t^{\ell}.
    \end{aligned}
    \end{equation}
    
Let $\St := [\s_t^0, \dots, \s_t^{m-1}]\in\mathbb{R}^{d\times m}$ and $\Yt := [\y_t^0, \dots, \y_t^{m-1}]\in\mathbb{R}^{d\times m}$ with $\s_t^{\ell} =\x_t^{\ell+1} - \x_t^{\ell}$ and $\y_t^{\ell} = \f^{\ell+1}_t -   \f^{\ell}_t$, for $\ell=0,\dots,m-1$.
When $\Yt$ is of full column rank, the constrained LS problem \eqref{equ:constrainedLS} can be written as an unconstrained LS problem, i.e.,
\begin{equation}
\label{equ:Yproblems}
 \z_t := \argmin_{\z \in \mathbb{R}^{m}} \| \Yt \z - \f_t\|^2  (= \Yt^{\dagger} \f_t), 
\end{equation}
and $\bfalpha_t$ can be recovered from $\z_t$ by setting $\alpha_t^0 = 1+z_t^0$, $\alpha_t^{\ell} = z_t^{\ell} - z_t^{\ell-1}$ for $0< \ell<m$, and $\alpha_t^m = -z_t^{m-1}$. 
This transformation leads to an equivalent multisecant quasi-Newton formulation of AAP, which updates $ \x_{t+1} $ based on an approximate Jacobian inverse $\Ht$. Specifically,
\begin{gather}
    \x_{t+1} \leftarrow \x_t - \Ht   \f_t, \\
    \begin{aligned}
        \text{where} \quad \quad & \Ht \leftarrow   -\damping \bfI + (\St+ \damping \Yt)(\Yt^T \Yt)^{-1} \Yt^T. 
        \label{equ:Hinv}
    \end{aligned}
\end{gather}
It is straightforward to verify that $\Ht$ satisfies the inverse multisecant equation $\Ht \Yt = \St$.

The following assumptions regarding the iterations generated by AAP are important for this paper. 

\begin{assumption}
\label{assu:syfullrank}
For all \(t \geq 0\), both \(\St\) and \(\Yt\) have full column rank.
\end{assumption}

\Cref{assu:syfullrank} is fundamental for all the results presented in this paper and will be assumed without further mention. This assumption is reasonable as we can adaptively adjust $m$ in each global iteration and stop the Picard iterations when the columns of $\St$ or $\Yt$ become linearly independent. Alternatively, we can remove the linearly dependent columns directly, though this approach may affect the analysis presented in this paper.

\begin{assumption}[Uniform Boundedness of $\cond(\St)$]
\label{assm:condst}
There exist constants  $T>0$ and $ M>0$ such that $\cond(\St)\leq M$ for all $t>T$.
\end{assumption}
\Cref{assm:condst} is a general assumption in the convergence analysis of AA (see Section~\ref{sec:review}).
We will show that $\St$ converges to a Krylov matrix in \Cref{sec:krylovbase}, 
and discuss this assumption in more detail in \Cref{sec:boundsing}.

\subsection{Multisecant matrices}
\label{sec:multi}

In this paper, we define the multisecant matrix \(\Bt\) as a matrix that satisfies the multisecant equation \( \Bt \St = \Yt \). 
It follows from 
$\s_t^{\ell} =\x_t^{\ell+1} - \x_t^{\ell} = \f^{\ell}_t$
that
\begin{equation}
\label{equ:multisecantaap}
    \Bt \St = \Yt,
    \quad \Longleftrightarrow \quad
    \Bt \f_t^{\ell} = \f_t^{\ell+1} - \f_t^{\ell}, \quad \forall \,\, 0 \leq \ell \leq m-1.
\end{equation}
Furthermore, we define the set
\begin{equation}
   \bbB := \{\bfB\in \bbR^{d \times d} \mid \bfB \St = \Yt   \}.
\end{equation}
It is evident that \(\bbB\) is non-empty and contains invertible matrices under \Cref{assu:syfullrank}. For instance, elements of \(\bbB\) include \(\Yt \St^{\dagger}\) and \(-\bfI + (\Yt + \St)\St^{\dagger}\).
Moreover, an invertible matrix in \(\bbB\) can be constructed by taking \(\bfY \bfS^{-1}\), where \(\bfY \in \mathbb{R}^{d \times d}\) and \(\bfS \in \mathbb{R}^{d \times d}\) are arbitrary invertible matrices with the first $m$ columns being \(\Yt\) and \(\St\), respectively.
For any $\Bt\in\bbB$, we call \( \widehat{f}_t(\x) = \Bt (\x - \x_t) + f(\x_t) \)  a multisecant linear approximation of $f$ at $\x_t$, which satisfies \( \widehat{f}_t(\x_t^{\ell}) = f(\x_t^{\ell}) \) for \( 0 \leq \ell \leq m \).
At global iteration $t$, this approximation is based on the $m$ Picard iteration points generated in AAP, as seen in the definition of $\bbB$.

\subsection{Existing convergence analysis of Anderson acceleration}
\label{sec:review}

The local \edit{$r$-linear} convergence rate of AA is established in~\cite{toth2015convergence,Kelley2018}. Subsequently,~\cite{Pollock2019, evans2020proof, pollock2021anderson} demonstrated how AA improves the convergence rate over fixed-point (Picard) iterations. The state-of-the-art convergence bound consists of a linear term and a higher-order term. A simplified result from~\cite[Theorem 5.1]{pollock2021anderson} is given by
$\| f(\x_{t+1}) \| \leq \kappa \theta_t \| f(\x_t) \| + c \sqrt{1 - \theta_t^2} \| f(\x_t) \| \sum_{\ell=0}^{m} \| f(\x_{t-\ell}) \|$ 
where \(\theta_t \leq 1\) is called the \textit{optimization gain}.
This result indicates that the local convergence rate of AA, \(\kappa \theta_t\), is superior to the local convergence rate of the Picard iteration, \(\kappa\).

The assumptions to guarantee the convergence of AA include a smoothness condition on \(g\) and a uniform boundedness assumption on the coefficients \(\bfalpha\)~\cite{toth2015convergence, evans2020proof, pollock2021anderson, ouyang2024descent}. 
Alternatives to the uniform boundedness assumption on \(\bfalpha\) is to assume sufficient linear independence among the columns of $\Yt$~\cite{pollock2021anderson} and the boundedness of the condition number of \(\St\) or \(\Yt\)~\cite{ouyang2024descent}. 
The equivalence between these conditions is discussed in \cite{ouyang2024descent}.
These conditions are also directly related to the conditioning of the constrained LS problem in AA.
In \cite{pollock2023filtering}, a filtering strategy is proposed to enforce these conditions by removing nearly linearly dependent columns from the LS problem.

An important property of AA is its equivalence to GMRES on linear problems.
When $g(\x) = \bfA\x+\b$,
\cite{walker2011anderson} shows that AA with no truncation solving $\x=g(\x)$ is ``essentially equivalent'' to GMRES applied to $(\bfI-\bfA)\x = \b$, in the sense that 
the weighted sum of the history points in each AA step equals the GMRES iterate $\x_t^{\mathrm{GMRES}} $ and the AA iterate $\x_{t+1}^{\mathrm{AA}}=g\left(\x_t^{\mathrm{GMRES}}\right)$.
It is also mentioned in~\cite{walker2011anderson} that a ``restarted'' variant of AA, in which the method proceeds without truncation for $m$ steps and then is restarted, is equivalent to GMRES($m$) applied to $(\bfI-\bfA)\x = \b$ followed by a fixed-point iteration.
In~\cite{lupo2019convergence}, the equivalence between Alternating Anderson-Richardson and GMRES is established. For nonlinear problems, AA is closely related to the Nonlinear Generalized Minimal Residual Method (NGMRES)~\cite{sterck2021asymptotic}.
This work will establish the equivalence between AAP and a multisecant-GMRES method for nonlinear case.

\section{Equivalence between AAP and multisecant-GMRES}
\label{sec:aapequivalence}

This section gives the equivalence between AAP and a multisecant-GMRES method.
We \edit{first show in \Cref{theorem:aapequivalence}} that, at global iteration $t$, AAP effectively
solves a multisecant linear system, followed by one additional linearized Picard iteration.
\edit{This equivalence leads to \Cref{coro:praap}, which provides an expression of the AAP iterates in terms of updates from the multisecant-GMRES method. This expression becomes useful in the analysis in Sections~\ref{sec:relation} and \ref{sec:conv}.}

\begin{theorem}[Equivalence between AAP($m$) and multisecant-GMRES($m$)]\label{theorem:aapequivalence}
    At global iteration $t$, let $\{\x_t^\ell\}_{\ell=0}^m$ be generated by AAP($m$), then, for any $\Bt\in\bbB$,
    \begin{equation}
        \label{equ:hatp}
        \sum_{\ell=0}^m \alpha_t^{\ell} \x_t^{\ell} = \x_t - \widehat{\p}_t,
        \qquad \text{where} \quad
 \widehat{\p}_t = \argmin_{\p \in \mathcal{K}_{\E}(\Bt, \f_t)} \| \Bt\p- \f_t \|^2.
    \end{equation}
 Further, $\x_{t+1} = \widehat{g}_t(\x_t - \widehat{\p}_t)$ where $\widehat{g}_t(\x) := (1-\damping)\x + \damping[ (\bfI+\Bt) (\x - \x_t) + g(\x_t)] $  is a (damped) linear approximation of $g$.
\end{theorem}

\begin{proof}
Let $\Bt\in\bbB$ be any multisecant matrix.
By expressing $\bfalpha$ in terms of $\z$ as given in \eqref{equ:Yproblems}, 
    \[
    \begin{aligned}
        \sum_{\ell=0}^m \alpha^{\ell} \f_t^{\ell} &= \f_t^0 - z^0(\f_t^1-\f_t^0) -\cdots- z^{m-1} (\f_t^m-\f_t^{m-1}) =\f_t - \Yt \z.
    \end{aligned}    
    \]
From \eqref{equ:multisecantaap}, $\Bt \f^{\ell}_t = \f^{\ell+1}_t - \f^{\ell}_t$ for $0\leq \ell \leq m-1$.
This implies that the residual sequence $\{\f^{\ell}_t\}_{\ell=0}^{m}$ is a Krylov sequence, i.e.,
\begin{equation}
\label{equ:aapkrylovbasis}
        \f^{\ell+1}_t = (\bfI+\Bt)\f^{\ell}_t = \cdots= (\bfI+\Bt)^{(\ell+1)}\f_t\in \mathcal{K}_{\ell+2}(\bfI+\Bt, \f_t).
\end{equation}
Then, the columns of $\Yt$ 
take the form
$\y_t^{\ell}= \f^{\ell+1}_t - \f^{\ell}_t=(\bfI+\Bt)^{(\ell+1)}\f_t - (\bfI+\Bt)^{(\ell)}\f_t  = \Bt (\bfI+\Bt)^{(\ell)}\f_t$, and thus
\[
 \spans(\Yt) =\spans\{ \y_t^0, \y_t^1,\cdots, \y_t^{\E-1} \} 
=\Bt\mathcal{K}_{\E}(\bfI+\Bt, \f_t) 
=  \Bt \mathcal{K}_{\E}(\Bt, \f_t),
\]
where the last equality is due to the shift invariance property of the Krylov subspace.
Therefore, we have the equivalence between the following three   problems:
\begin{equation}
\label{equ:threeproblems}
    \min_{\bfalpha \in \mathbb{R}^{m+1},
    \atop
    \sum_{\ell=0}^{m} \alpha^{\ell} = 1} \|\sum_{\ell=0}^{m} \alpha^{\ell} \f_t^{\ell}\|^2 
\quad \Leftrightarrow \,\, \min_{\z \in \mathbb{R}^{m}} \| \Yt \z - \f_t \|^2 
\quad \Leftrightarrow \,\, \min_{\p \in \mathcal{K}_{\E}(\Bt, \f_t)} \|\Bt \p - \f_t\|^2.
\end{equation}
It is straightforward to verify that the solutions to each of the problems satisfy 
\begin{equation}    \textstyle
\sum_{\ell=0}^m \alpha_t^{\ell} \f_t^{\ell}  = \f_t-\Yt \z_t=  \f_t -\Bt\widehat{\p}_t.
\label{equ:sulutionequivalence}
\end{equation}
As $\Bt \St = \Yt$, we have  $\Bt\widehat{\p}_t = \Yt \z_t = \Bt \St\z_t $ holds for any $\Bt \in \bbB$. Since $\bbB$ contains invertible elements, we have $ \widehat{\p}_t = \St\z_t$.
The equivalence between $\bfalpha_t$ and $\z_t$ then implies $\sum_{\ell=0}^{m} \alpha_t^{\ell} \x_t^{\ell} 
= \x_t - \St \z_t = \x_t - \widehat{\p}_t$, and proves \eqref{equ:hatp}.

Furthermore, let
$\widehat{f}_t$ be the multisecant linear approximation of $f$ at $\x_t$ as
\( \widehat{f}_t(\x) = \Bt (\x - \x_t) + f(\x_t) \).
Then,  $\widehat{g}_t(\x) = (1-\damping)\x + \damping(\x+ \widehat{f}_t(\x)  )$.
In addition,
for all $0\leq \ell \leq \E$,
$g (\x_t^{\ell}) = \x_t^{\ell}+f (\x_t^{\ell})$
and $f(\x_t^{\ell}) =\widehat{f}_t(\x_t^{\ell}) $.
It follows that
\[\textstyle
\sum_{\ell=0}^{m} \alpha_t^{\ell} g (\x_t^{\ell})
 =\sum_{\ell=0}^{m} \alpha_t^{\ell} \x_t^{\ell}+ \sum_{\ell=0}^{m} \alpha_t^{\ell} \widehat{f}_t (\x_t^{\ell}) 
 =\sum_{\ell=0}^{m} \alpha_t^{\ell} \x_t^{\ell}+ \widehat{f}_t (\sum_{\ell=0}^{m} \alpha_t^{\ell} \x_t^{\ell}) ,
\]
where the last equation is due to the linearity of $\widehat{f}_t$.
Finally, we have $
\x_{t+1} = (1-\damping) \sum_{\ell=0}^{m} \alpha_t^{\ell} \x_t^{\ell} + \damping [\sum_{\ell=0}^{m} \alpha_t^{\ell} \x_t^{\ell}+ \widehat{f}_t (\sum_{\ell=0}^{m} \alpha_t^{\ell} \x_t^{\ell}) ] = \widehat{g}_t( \sum_{\ell=0}^{m} \alpha_t^{\ell} \x_t^{\ell} ) = \widehat{g}_t(\x_t - \widehat{\p}_t)$,
which completes the proof.
\end{proof}

We have shown that the LS problem \eqref{equ:constrainedLS} in AAP($m$) is equivalent to applying GMRES($m$) to the multisecant linear system $\Bt\p=\f_t$ for any $\Bt \in \bbB $.
We call $\widehat{\p}_t$ the \textit{multisecant-GMRES direction}, and $\widehat{\p}_t $ is equal to $ \St \Yt^{\dagger} \f_t$.
As shown in \cref{theorem:aapequivalence}, 
the weighted average of the history points equals the multisecant-GMRES update $\x_t - \widehat{\p}_t$.
The above theorem holds for any $\Bt\in \bbB$ because they generate the same Krylov (residual) subspaces
\begin{equation}
    \label{equ:pbp}
    \mathcal{K}_{m}(\Bt, \f_t) = \spans(\St),
\qquad
\Bt\mathcal{K}_{m}(\Bt, \f_t) = \spans(\Yt),
\quad 
\forall \, \Bt \in \bbB.
\end{equation}
We refer to \cite{greenbaum1994matrices} for the properties of
matrices that generate the same Krylov residual spaces.

The equivalence between AAP and multisecant-GMRES resembles the one given in \cite{walker2011anderson} for linear $g$. 
A similar result was also mentioned for a restarted variant of AA applied on gradient descent~\cite{ouyang2020anderson}. 
The equivalence between other variants of AA and the multisecant methods is discussed in \cite{xue2024}.
 \medskip

Let us briefly take a closer look at the damping ratio. To that end,
we define the \textit{multisecant-GMRES residual} at iteration $t$ as $\widehat{\r}_t$.
According to the equivalence between the three minimization problems in \eqref{equ:threeproblems}, we have
\begin{equation}
\label{residualequation}
  \widehat{\r}_t = \Bt \widehat{\p}_t - \f_t = \Yt\z_t-\f_t= - \sum_{\ell=0}^m \alpha_t^{\ell} \f_t^{\ell}  .
\end{equation}
The following corollary shows that the damping ratio acts as a stepsize parameter in the direction of the residual.

\begin{corollary}\label{coro:praap}
The iterate $\x_{t+1}$ generated from \eqref{equ:AAPupdate} in AAP satisfies
\begin{equation}
\x_{t+1}=    \x_t - \widehat{\p}_t - \damping \widehat{\r}_t.
\end{equation}
\end{corollary}
\begin{proof}
    The result follows from $
    \x_{t+1}  = \sum_{\ell=0}^{m} \alpha_t^{\ell}  \x_{t}^{\ell}  + \damping     \sum_{\ell=0}^m \alpha_t^{\ell}       \f_t^{\ell} 
      = \x_t - \widehat{\p}_t - \damping \widehat{\r}_t.$
\end{proof}
When $\damping$ satisfies $\| \bfI - \damping  \Bt\| \leq 1$, the ``damped'' direction $\widehat{\p}_t + \damping \widehat{\r}_t$ could be a better solution to the multisecant linear system $\Bt\p=\f_t$ than the original direction $\widehat{\p}_t$. Indeed, in this case,
$\| \Bt (\widehat{\p}_t + \damping \widehat{\r}_t) -\f_t\|    
 = \| (\bfI - \damping  \Bt) (\Bt \widehat{\p}_t  -\f_t)  \| 
\leq  \|  \Bt \widehat{\p}_t  -\f_t \|$.

\section{Connection Between AAP and Newton-GMRES}
\label{sec:relation}

Newton's method finds the roots of a nonlinear function \( f(\x) \) by using its Jacobian information. At iteration \( t \), a linear approximation of \( f \) at \( \x_t \) is given by
\begin{equation}
\label{equ:taylor}
    f(\x) \approx \Jt (\x - \x_t) + f(\x_t),
\end{equation}
where $\Jt:= f'(\x_t)$ is the Jacobian of \( f \) at \( \x_t \).
The Newton-GMRES(\(m\)) method is a variant of the inexact-Newton method that uses GMRES(\(m\)) to approximately solve this linearized system for a search direction $\p^{\textup{N}}_t$ at each iteration as follows
        \begin{equation}
        \x_{t+1} := \x_{t} - \p^{\textup{N}}_t\quad\text{with}\quad
      \p^{\textup{N}}_t  :=   \argmin_{ \p \in \mathcal{K}_m( \Jt, \f_t ) } \|\Jt \p - \f_t  \|.
        \end{equation}
However, the Jacobian is not always available or easy to compute directly.

As shown in \Cref{coro:praap}, at iteration \( \x_t \), the search direction given by AAP is equivalent to the multisecant-GMRES direction based on $\Bt \in \bbB$ with an additional damping term. 
In this section, we establish the connection between AAP and the Newton-GMRES method by showing that, (i) when the residual $\|\f_t\|$ goes to zero as $t\to\infty$, the distance between the set of multisecant matrix $\bbB$ given by AAP and the Jacobian matrix $\Jt$ converges to zero (Section~\ref{sec:EEE}), (ii) the matrices $\St$, $\Yt$, and $\Ht$ constructed by AAP converge respectively to their counterparts in the Newton-GMRES method (Section~\ref{sec:krylovbase}), and (iii) the optimization gain $\theta_t$ in AAP converges to an analogous quantity in Newton-GMRES (Section~\ref{sec:opop}).
\edit{The results in Section~\ref{sec:EEE} are used later to show the one-step bound of the residual in \Cref{thm:residual_cvgce}, which 
leads to the $q$-linear local convergence result of AAP in \Cref{theorem:localconv}, whereas the convergence to Newton-GMRES components proved in Sections~\ref{sec:krylovbase} and \ref{sec:opop} become essential when estimating the asymptotic $q$-linear convergence factor in \Cref{cor:asymptotic-qfactor}.}

One of the key properties of AAP that enables the analysis in this section is the $m$ Picard iteration steps taken in a global iteration.
The Picard iteration steps often progress more conservatively than the acceleration step. Therefore, the multisecant matrix given by these history points could be a good approximation to the local Jacobian $\Jt$.
In fact, the following lemma provides bounds on the Picard iteration updates and residuals, which are essential in the analysis in this section.
\begin{lemma}
\label{lemma:picardpoints}
Under \Cref{assumption:gaap}(a), at global iteration $t$ of AAP, for all $0 \leq \ell \leq \E$, 
\[
\|\f^{\ell}_t\| \leq \|\f_t\|,
\quad \text{and} \quad
\| \x^{\ell}_t - \x_t\| \leq \ell\|\f_t\|.
\]
\end{lemma}

\begin{proof}
Since $\x_t^{\ell} = g(\x_t^{\ell-1})$ and $\x_t^0 = \x_t$, we have $\x_t^{\ell} = \x_t^{\ell-1} + f(\x_t^{\ell-1})$
and thus $
    \|\x_t^{\ell} - \x_t\| = \|\sum_{i=0}^{\ell-1}f(\x_t^{i})\| \leq \sum_{i=0}^{\ell-1}\|f(\x_t^{i})\|$.
The non-expansive property of $g$ gives that, for $i=1,\dots,\ell$,
$$ \|f(\x_t^{i})\| = \|g(\x_t^{i}) - \x_t^{i}\| = \|g(\x_t^{i}) - g(\x_t^{i-1})\| \leq \contraction \|\x_t^{i} - \x_t^{i-1}\| = \contraction \|f(\x_t^{i-1})\|.$$ 
Since $\contraction\leq 1$, $\|\f^{\ell}_t\| \leq \|\f_t\|$ and $
    \|\x_t^{\ell} - \x_t\| \leq  \sum_{i=0}^{\ell-1} \|f(\x_t)\| \leq \ell\|\f_t\|$. 
\end{proof}
Another useful technique in the following analysis is the fundamental theorem for line integrals, i.e.,
\begin{equation}\label{eqn:line_int}\textstyle
    f(\y) - f(\x) = \int_0^1 f'(\x + r(\y-\x)) (\y-\x) \, \d r.
\end{equation}
With these results, we show in the following subsections that, as the AAP residual $\f_t$ approaches zero, \(\Bt\), \(\St\), \(\Yt\), \(\Ht\), and \(\theta_t\) in AAP converge to their counterparts in Newton-GMRES.

\subsection{Convergence of the multisecant matrix}
\label{sec:EEE}

We start the analysis by showing that the distance between the set of multisecant matrices $\bbB$ generated by AAP and the Jacobian $\Jt$ goes to zero provided that the residual approaches zero as $t\to\infty$.
Here, the distance between $\bbB$ and $\Jt$ is defined as the norm of
\begin{equation}
\label{definition:et}
    \Et := \Bt^*-\Jt,
    \quad  \text{where} \quad
     \Bt^* := \argmin_{\bfB \in \bbB } \|\bfB-\Jt\|.
\end{equation}
Because $\St$ is full column rank, the minimal-norm property~\cite{walker2011anderson} gives
\begin{equation}    \label{equ:Eaa}
\Et = (\Jt \St-\Yt)\St^\dagger.
\end{equation}
Note that $ \Jt+\Et = \Bt^*\in\bbB$.
In addition, when $g$ is linear, $\Jt$ satisfies $\Jt \St = \Yt$ and thus $\| \Et\|=0$.
An upper bound on $\|\Et\|$ for the nonlinear case is provided in the following theorem.

\begin{theorem}\label{theorem:etbound}    
Under \Cref{assumption:gaap},    
\begin{equation}
\label{equ:etbound}
        \| \Et\| \leq \gamma \E^{\frac{3}{2}} \cond(\St) \|\f_t\|.
    \end{equation}
Furthermore, if \Cref{assm:condst} holds, 
then 
\begin{equation}
\label{equ:etbound2}
    \| \Et\| \leq C_E\|\f_t\|, \quad \quad \forall\,\,t>T,
\end{equation}
where $C_E>0$ is a constant that is independent of $t$.
\end{theorem}

\begin{proof}
It follows from \eqref{equ:Eaa} that
\[
\|\Et \| \leq  \|\Yt - \Jt \St\| \|\St^\dagger\|.
\]
Let \(\Yt - \Jt \St = [\v^0, \v^1, \dots, \v^{m-1}] \in \mathbb{R}^{d \times m}\).
For \(0 \leq \ell \leq m-1\), 
\[
\v^{\ell} = \y^{\ell} - \Jt \s^{\ell} = f(\x_{t}^{\ell+1}) - f(\x_t^{\ell}) - f'(\x_t)(\x_{t}^{\ell+1} - \x_{t}^{\ell}).
\]
By using \eqref{eqn:line_int} and \(\s_t^{\ell} = \x_{t}^{\ell+1} - \x_{t}^{\ell} = \f_{t}^{\ell}\),
\[
\begin{aligned}
    \v^{\ell} &= \int_0^1 f'(\x^{\ell}_t + r\f^{\ell}_t) \f^{\ell}_t \, \d r - f'(\x_t)\f^{\ell}_t = \int_0^1 [f'(\x^{\ell}_t + r\f^{\ell}_t) - f'(\x_t)] \f^{\ell}_t \, \d r.
\end{aligned}
\]
Applying \(f'(\x) = g'(\x) - \bfI\) and the Lipschitz continuity of \(g'\) gives, for $r\in[0,1]$,
\[
\| f'(\x^{\ell}_t + r\f^{\ell}_t) - f'(\x_t) \| \leq \gamma \| \x^{\ell}_t + r\f^{\ell}_t - \x_t \| \leq \gamma( \| \x^{\ell}_t - \x_t \| + \|\f^{\ell}_t \|).
\]
By \Cref{lemma:picardpoints}, $
\| \x^{\ell}_t - \x_t \| + \|\f^{\ell}_t \| \leq (\ell+1) \|\f_t\| \leq \gamma m \|\f_t\|$ as $\ell < m$,
which leads to
\begin{equation}
    \label{equ:errorfj}
    \| \v^{\ell} \| \leq \int_0^1 \| f'(\x^{\ell}_t + r\f^{\ell}_t) - f'(\x_t) \| \|\f^{\ell}_t \| \, \d r \leq \gamma m \|\f_t\| \|\s_t^{\ell}\|.
\end{equation}
Using the matrix norm inequality \(\| \cdot \| \leq \| \cdot \|_F\),
\[
\| \Jt \St - \Yt \| \leq \| \Jt \St - \Yt \|_F = \sqrt{\sum_\ell \| \v^{\ell} \|^2} \leq \gamma m \|\f_t\| \sqrt{\sum_\ell \| \s_t^{\ell} \|^2} = \gamma m \|\f_t\| \|\St\|_F.
\]
Since \(\St\) has full column rank and \(\|\St\|_F \leq \sqrt{m} \|\St\|\), it follows that
\[
\|\Et \| \leq \sqrt{m} \gamma m \|\f_t\| \|\St\| \|\St^\dagger\| = \gamma m^{\frac{3}{2}} \|\f_t\| \cond(\St).
\]
The second claim is then a direct consequence of \Cref{assm:condst}.
\end{proof}

The bound given in \eqref{equ:etbound2} shows that $\|\Et\| \leq \mathcal{O}(\|\f_t\|)$, which leads to the following convergence result.

\begin{corollary}
\label{coro:btconver}
Under Assumptions~\ref{assumption:gaap} and \ref{assm:condst},
suppose that $\f_t\to 0$ as $t\to\infty$, then, as $t\to\infty$, 
\[     \Et \rightarrow 0 \quad\text{and}\quad \Bt^* \rightarrow \Jt,
    \]
    where $\Bt^* \in \bbB$ is defined in \eqref{definition:et}.
\end{corollary}

The convergence of $\Bt $ to $ \Jt$ suggests that the multisecant direction $\widehat{\p}_t$ could converge to the Newton-GMRES direction ${\p}^{\textup{N}}_t:=\argmin_{\p \in \mathcal{K}_{m}(\Jt, \f_t)} \| \Jt \p - \f_t \|$, which we prove later in \Cref{theorem:pconvergence}.

\subsection{Limits of $\St$, $\Yt$, and $\Ht$}
\label{sec:krylovbase}

In this section, we establish the limits of \(\St\) and \(\Yt\) as $t\to\infty$ under the assumption that \(\f_t \) approaches zero. These limits are related to a Krylov matrix spanned by \( g'(\x_t) \) and \( \f_t \), which further implies the convergence of the approximate Jacobian inverse \( \Ht \) of AAP defined in \eqref{equ:Hinv}.

We first review the Newton-GMRES method.
At iteration $\x_t$,
the Krylov matrix associated with $\mathcal{K}_m(\Jt, \f_t)$ can be written as
\begin{equation}
    \Ft :=[ \f_t,  \Jt^{(1)} \f_t, \Jt^{(2)} \f_t, \cdots, \Jt^{(\E-1)} \f_t  ] \in \bbR^{d \times \E}.
\end{equation}
The Newton-GMRES direction $\p^{\textup{N}}_t := \argmin_{ \p \in \mathcal{K}_m( \Jt, \f_t ) } \|\Jt \p - \f_t  \|$ satisfies $\Jt \p^{\textup{N}}_t $$ = \proj_{ \Jt \Ft}( \f_t)$ where $\proj_{\bfA} = \bfA\bfA^{\dagger}$ denotes the orthogonal projection onto the subspace spanned by the columns of $\bfA$.
Define 
\begin{equation}
    \Gt := [ \f_t,  g'( \x_t)^{(1)} \f_t, g'( \x_t)^{(2)} \f_t, \cdots, g'( \x_t)^{(\E-1)} \f_t  ] \in \bbR^{d \times \E}.
\end{equation}
Since $g'( \x_t) = \bfI+ \Jt$, the shift invariance of the Krylov subspace leads to $\proj_{ \Gt} = \proj_{ \Ft}$.
Therefore, $ \proj_{ \Jt \Gt} =  \proj_{ \Jt \Ft}$.
If $\Jt$ is invertible, we have
\begin{equation}
    \label{equ:newtonproj}    
 \p^{\textup{N}}_t = \Jt^{-1}\proj_{ \Jt \Gt}( \f_t).
\end{equation}
We refer to $\Jt^{-1}\proj_{ \Jt \Gt}$ as the \textit{Newton-GMRES update operator}.

As shown in \eqref{equ:aapkrylovbasis}, the AAP residuals $\{\f_t^{\ell}\}_{\ell=0}^{m-1}$ form a basis of the Krylov subspace $\mathcal{K}_m(\Bt, \f_t)$ as
$\f_t^{\ell} = (\bfI+\Bt)^{(\ell)}\f_t$.
Therefore,
\[
\St =
[ \f_t^0, \f_t^1, \cdots,  \f_t^{m-1} ]
=[ \f_t,  (\bfI+\Bt)^{(1)}\f_t, (\bfI+\Bt)^{(2)}\f_t, \cdots, (\bfI+\Bt)^{(\E-1)}\f_t  ] 
\]
is a Krylov matrix associated with $\bfI+\Bt$ and $\f_t$.
As $\f_t$ approaches zero, the matrices \(\St\), \(\Yt\), and \(\Gt\) all approach zero.
Further, the following lemma gives estimates of the distances between $\St$, $\Yt$ and their Newton-GMRES counterparts.

\begin{lemma}
    \label{lemma:convStYt}
Under Assumption \ref{assumption:gaap},
\begin{equation}
\label{equ:ytstbound}
     \|\St - \Gt\|   \leq \gamma m^{\frac{5}{2}} \|\f_t\|^2 
\qquad \text{and} \qquad
 \|\Yt - \Jt \,  \Gt\|\leq 2\gamma m^{\frac{5}{2}} \|\f_t\|^2.
\end{equation}
\end{lemma}

\begin{proof}
For a better presentation, we omit the subscript \(t\) in this proof unless necessary. Note that \(\x^0 = \x_t\) is the starting point for the Picard iterations at the global iteration \(t\),
and $\f^{\ell} = f(\x_t^{\ell})$ for all $0\leq \ell \leq m-1$.

We first define \(\Deltat = \Gt - \St\) and write \( \Deltat = [\bfepsilon^{0}, \bfepsilon^{1}, \bfepsilon^{2}, \cdots, \bfepsilon^{m-1}]  \).
By the definitions of $\Gt$ and $\St$, $\bfepsilon^{0}=0$ and $\bfepsilon^\ell=\f^\ell - g^\prime(\x_t)^{(\ell)}\f_t$ for $\ell=0,\dots, m-1$.
We next estimate $\bfepsilon^\ell$.
By \eqref{eqn:line_int}, \(f(\x^{\ell}) - f(\x^{\ell-1}) = \int_0^1 f'(\x^{\ell-1} + r\f^{\ell-1}) \f^{\ell-1} \, \d r\) as $\x^{\ell} - \x^{\ell-1} = \f^{\ell}$.
Thus,
\[
\begin{aligned}
\f^{\ell} =    f(\x^{\ell}) &\textstyle
= \int_0^1 [\bfI + f'(\x^{\ell-1} + r\f^{\ell-1})] \f^{\ell-1} \, \d r\\
    &\textstyle
    = \int_0^1 g'(\x^{\ell-1} + r\f^{\ell-1}) \f^{\ell-1} \, \d r\\
    &\textstyle
    = g'(\x_t) \f^{\ell-1} + \int_0^1 [g'(\x^{\ell-1} + r\f^{\ell-1}) - g'(\x_t)] \f^{\ell-1} \, \d r\\
    &= \bfK \f^{\ell-1} + \v^{\ell-1},
\end{aligned}
\]
where \(\bfK := g'(\x_t)\) and  
\(\v^{\ell-1} := \int_0^1 \left[g'(\x^{\ell-1} + r\f^{\ell-1}) - g'(\x_t)\right] \f^{\ell-1} \, \d r
\).
Applying this relation iteratively gives $\f^{\ell} = \bfK^{(\ell)} \f_t + \sum_{i=0}^{\ell-1} \bfK^{(i)} \v^{\ell-1 - i}$.
Therefore, \(\bfepsilon^{\ell} = \sum_{i=0}^{\ell-1} \bfK^{(i)} \v^{\ell-1 - i}\).
Furthermore, under Assumption \ref{assumption:gaap}, \(\|\bfK\| \leq \contraction \leq 1\) and
\[
\|\v^{\ell-1}\| \leq \gamma \|\x^{\ell-1} + r\f^{\ell-1} - \x_t\| \|\f^{\ell-1}\| \leq \gamma \ell \|\f_t\| \|\f^{\ell-1}\| \leq m \gamma \|\f_t\|^2,\]
where the second inequality follows from \Cref{lemma:picardpoints} and $0\leq r\leq 1$.
Then,
  \[\textstyle\|\bfepsilon^{\ell}\| = \sum_{i=0}^{\ell-1} \bfK^{(i)} \v^{\ell-1 - i}\leq \sum_{i=0}^{\ell-1} \|\v^{\ell-1 - i}\| \leq  \ell (m \gamma \|\f_t\|^2)
  \leq m^2 \gamma \|\f_t\|^2.\]
Therefore, 
\begin{equation}\textstyle
\label{equ:Stbound}
     \| \St -  \Gt\| = \|\Deltat\| \leq \|\Deltat\|_F = \sqrt{\sum_{i=0}^{m-1} \| \bfepsilon^{i}\|^2} \leq \gamma m^{\frac{5}{2}} \|\f_t\|^2.
\end{equation}

To estimate $\|\Yt-\Jt\Gt\|$, we use the relations established above and write the $\ell$-th column of $\Yt$ as \[
        \y_t^{\ell} = \f^{\ell+1} - \f^{\ell} = \bfK^{(\ell+1)} \f_t + \bfepsilon^{\ell+1} - \bfK^{(\ell)} \f_t - \bfepsilon^{\ell} = \Jt 
 \bfK^{(\ell)} \f_t + \bfepsilon^{\ell+1} - \bfepsilon^{\ell},\] 
where the last equality is due to 
\(\bfK = g'(\x_t) = \bfI+ f'(\x_t) = \bfI+\Jt\).
Then, we write $\Yt = \Jt \Gt + \widetilde{\bfDelta}_t - \Deltat$ where
\( \widetilde{\bfDelta}_t := [\bfepsilon^{1}, \bfepsilon^{2}, \bfepsilon^{3}, \cdots, \bfepsilon^{m}] \in \mathbb{R}^{d \times m} \).
Following a similar approach as above, one can obtain $\|\widetilde{\bfDelta}_t\|\leq\gamma m^{\frac{5}{2}} \|\f_t\|^2$, and thus
\begin{equation}
\label{equ:ytbound}
    \left\| \Yt - \Jt \Gt \right\| \leq \|\widetilde{\bfDelta}_t\| + \|\Deltat\| \leq  
 2 \gamma m^{\frac{5}{2}} \|\f_t\|^2.
\end{equation}
\end{proof}

When \(\f_t\) approaches 0, despite \(\St\) and \(\Yt\) converging to 0, the above lemma shows that \(\frac{1}{\|\f_t\|}\St \rightarrow \frac{1}{\|\f_t\|}\Gt\) and \(\frac{1}{\|\f_t\|}\Yt \rightarrow \frac{1}{\|\f_t\|}\Jt \Gt\).
We note that, when $g$ is linear, i.e., $\gamma=0$, for all $t$,
\begin{equation}
    \label{equ:linearcase}
\St = \Gt,
\qquad \text{and}\qquad
\Yt = \Jt \Gt.
\end{equation}

\medskip
Now we are ready to give the convergence of the  approximate Jacobian inverse \(\Ht\) in AAP.
As given in \eqref{equ:Hinv},
$\Ht =  \St\Yt^{\dagger} - \damping (\bfI - \proj_{ \Yt})$ where $\damping$ is the damping ratio.
The definitions of $\widehat{\p}_t$ and $\widehat{\r}_t$ in \eqref{equ:hatp} and \eqref{residualequation} lead to $\widehat{\p}_t = \St\Yt^{\dagger} \f_t$ and $\widehat{\r}_t  = - \damping (\bfI - \proj_{ \Yt})\f_t$, respectively.
Thus, $\St\Yt^{\dagger}$ corresponds to the multisecant-GMRES update operator,
and $\bfI - \proj_{ \Yt}$ corresponds to the extra linearized fixed-point iteration in AAP.
The following corollary gives the convergence of $\St\Yt^{\dagger}$ to the Newton-GMRES update operator defined in \eqref{equ:newtonproj},
as well as the convergence of $\Ht$.

\begin{assumption}
\label{assum:ytgt}
There exist constants $T>0$ and $M>0$ such that, for all $t\geq T$, $\Jt$ is invertible, 
$\sigma_{\min}(  \Yt  ) \geq M \|\f_t\|$,
and $ \Jt \,  \Gt$ has full column rank.
\end{assumption}

\begin{theorem}\label{thm:Hcvgce}
Under Assumptions~\ref{assumption:gaap} and \ref{assum:ytgt},
when $\f_t$ goes to 0,
\begin{equation}
    \St\Yt^{\dagger} \rightarrow  \Jt^{-1}\proj_{ \Jt \Gt},
    \qquad
    \Ht \rightarrow \Jt^{-1}\proj_{ \Jt \Gt} - \damping( \bfI - \proj_{ \Jt \Gt}).
\end{equation}    
\end{theorem}

\begin{proof}
 When $\f_t$ goes to 0, \Cref{lemma:convStYt} gives
 $ \Yt  \rightarrow  \Jt \,  \Gt$.
Because both $\Yt$ and $\Jt \,  \Gt$ have full column rank, it follows from~\cite[Theorem~3.4]{stewart1977perturbation} that 
\begin{equation}
    \label{equ:ytinvbound}
    \|  \Yt ^{\dagger} - (\Jt \,  \Gt)^{\dagger}\| \leq \sqrt{2} \|  \Yt^{\dagger} \|\,\| (\Jt \,  \Gt)^{\dagger}\|\, \left\| \Yt  - \Jt \,  \Gt \right\|.
\end{equation}
Furthermore, $\|  \Yt^{\dagger}\| = \frac{1}{ \sigma_{\min}(  \Yt  )} \leq \frac{1}{M\|\f_t\|}$ and, for sufficiently small $\|\f_t\|$,
\begin{equation}
\label{equ:deltaM}
    \| ( \Jt \,  \Gt )^{\dagger}\| =  \frac{1}{\sigma_{\min}( \Jt \,  \Gt  )} \leq  \frac{1}{\sigma_{\min}( \Yt  )-\|\Yt - \Jt\Gt \|} \leq \tau \frac{1}{ M\|\f_t\|},
\end{equation}
where $\tau<1$ is a constant and the first inequality follows from \cite[Corollary 2.4.4]{golub2013matrix} and the estimate in \Cref{lemma:convStYt} on $\|\Yt - \Jt\Gt \|$.
With these estimates, the RHS of \eqref{equ:ytinvbound} is uniformly bounded and thus, as $\|\f_t\|\to0$,
\begin{equation} \textstyle
    \label{equ:ytinvboundscaked}
    \|  (\frac{1}{ \|\f_t\| }\Yt) ^{\dagger} - (\frac{1}{ \|\f_t\| }\Jt \,  \Gt)^{\dagger}\| = \|\f_t\| \|  \Yt ^{\dagger} - (\Jt \,  \Gt)^{\dagger}\|\to 0.
\end{equation}
Therefore, since $\Jt $ is invertible, $\Gt = \Jt^{-1} \Jt \Gt$, and then
\[ \textstyle
    \St\Yt^{\dagger} 
= (\frac{1}{ \|\f_t\| }\St) (\frac{1}{ \|\f_t\| } \Yt )^{\dagger}
\quad \rightarrow  \quad
\Jt^{-1} (\frac{1}{ \|\f_t\| } \Jt \,  \Gt) (\frac{1}{ \|\f_t\| } \Jt \,  \Gt)^{\dagger} =
\Jt^{-1} \proj_{ \Jt \Gt}.
\]
Further,
$\proj_{ \Yt } = \Yt \Yt ^{\dagger}  \rightarrow  (\Jt \Gt)(\Jt \Gt)^{\dagger} =\proj_{ \Jt \Gt} $.
It follows that
\[
\Ht =  \St\Yt^{\dagger} - \damping (\bfI - \proj_{ \Yt})
\quad \rightarrow  \quad
\Jt^{-1}\proj_{ \Jt \Gt} - \damping( \bfI - \proj_{ \Jt \Gt}),
\]
which completes the proof.
\hfill $\qed$
\end{proof}

When \(\damping = 0\), the above theorem demonstrates the convergence of the approximate Jacobian inverse \(\Ht\) to the Newton-GMRES update operator. This suggests that the iteration sequences generated by AAP and Newton-GMRES are likely to be close to each other, provided that they start from the same initial point with a sufficiently small residual.

As $\widehat{\p}_t = \St\Yt^{\dagger}\f_t$,
the following convergence of $\widehat{\p}_t$ to the Newton-GMRES search direction $\p^N_t=\Jt^{-1}\proj_{ \Jt \Gt}( \f_t)$ is straightforward.

\begin{corollary}
\label{theorem:pconvergence}
Under Assumptions~\ref{assumption:gaap} and \ref{assum:ytgt},
when $\f_t$ goes to 0, we have  $\frac{1}{ \|\f_t\|}\widehat{\p}_t \rightarrow \frac{1}{ \|\f_t\|} \p^N_t$.
\end{corollary}

\begin{remark}
    For \Cref{assum:ytgt}, the condition that $\Jt$ is invertible is satisfied when $g$ is a contractive mapping, i.e., $\contraction<1$.
\end{remark}

\subsection{Convergence of the optimization gain}
\label{sec:opop}

As mentioned in Section~\ref{sec:review}, the optimization gain $\theta_t$, which is defined as the ratio between the minimum of the constrained LS problem \eqref{equ:constrainedLS} and $\|\f_{t}\|$, is a crucial factor in the convergence analysis of Anderson acceleration.
Based on \eqref{residualequation},
the optimization gain can be written as
\begin{equation}
        \label{equ:thetaaap}
    \theta_t: = \frac{\| \sum_{\ell=0}^m \alpha_t^{\ell} \f_t^{\ell}\|  }{\|\f_t\|} 
    = \frac{\|\Bt\widehat{\p}_t - \f_t\|}{\|\f_t\|}
    = \frac{ \|\Yt \z_t -\f_t\|  }{\|\f_t\|}
    = \frac{ \| (\bfI-\proj_{\Yt}) (\f_t)\|}{\|\f_t\|},    
\end{equation}
In this section, we provide convergence estimates for $\theta_t$ to an analogous gain term in the Newton-GMRES method.

Recall that at iterate \(\x_t\), the Newton-GMRES method uses GMRES($m$) to solve \(\Jt \p - \f_t=0\)  where $\Jt = f'(\x_t)$ is the Jacobian at $\x_t$.
We define the \textit{Newton-GMRES(\(m\)) gain} as \(\|\r_t^J\|/\|\f_t\|\) where $\r_t^J$ is the GMRES residual:
\begin{equation}
\label{definition:gain}
\|\r_t^J\| := \min_{\p \in \mathcal{K}_{m}(\Jt, \f_t)} \|\Jt \p - \f_t\| = \|(\bfI- \proj_{\Jt\Gt}) (\f_t)\|,
\end{equation}
as shown in \eqref{equ:newtonproj}.  
The term \textit{Newton-GMRES(\(m\)) gain}
is used to distinguish it from the Newton-GMRES algorithm
since $\r_t^J$ is evaluated at iteration $\x_t$ generated by the AAP method.
If \(g\) is a linear function,  we have \(\theta_t = \|\r_t^J\|/\|\f_t\|\) according to \eqref{equ:linearcase}.
The following theorem provides the convergence of $\theta_t$ to $\|\r_t^J\| / \|\f_t\|$ provided that $\|\f_t\|\to0$ as $t\to\infty$.

\begin{theorem}[Limit of optimization gain]
    \label{theorem:gain}
Under Assumptions~\ref{assumption:gaap} and \ref{assum:ytgt},
when $\|\f_t\|$ is small enough, 
\begin{equation}
    \label{equ:optimizationgainbound}
    \left|\, \theta_t  - \|\r_t^J\|/\|\f_t\| \,\right| \leq \mathcal{O}(\|\f_t\|),
\end{equation}
implying that \(\theta_t\) converges to \(\|\r_t^J\|/\|\f_t\|\) as \( \|\f_t\| \) approaches zero.
\end{theorem}
\begin{proof}
By the definitions,
\begin{equation}\textstyle
\label{equ:thetatbound}
    \left| \theta_t  - \frac{\|\r_t^J\|}{\|\f_t\|} \right| 
    =  \frac{\|  (\bfI- \proj_{\Yt}) (\f_t) - (\bfI-\proj_{\Jt\Gt}) (\f_t)\|}{\|\f_t\|}
    \leq \|   \proj_{\Yt}  - \proj_{\Jt\Gt} \|.
\end{equation}  

\Cref{lemma:convStYt} shows that $\|\frac{1}{\|\f_t\|}\Yt  - \frac{1}{\|\f_t\|}\Jt\Gt\| \leq \mathcal{ O} (\|\f_t\| )$.
Meanwhile, when $\|\f_t\|$ is small enough,
\eqref{equ:ytinvboundscaked} in \Cref{thm:Hcvgce} shows that 
$\|(\frac{1}{\|\f_t\|}\Yt)^{\dagger}  - (\frac{1}{\|\f_t\|}\Jt\Gt)^{\dagger}\| \leq \mathcal{ O} (\|\f_t\| )$,
Therefore,
\[
\|   \proj_{\Yt}  - \proj_{\Jt\Gt} \| = \|(\frac{1}{\|\f_t\|}\Yt)(\frac{1}{\|\f_t\|}\Yt)^{\dagger}  - (\frac{1}{\|\f_t\|}\Jt\Gt)(\frac{1}{\|\f_t\|}\Jt\Gt)^{\dagger}\| \leq \mathcal{ O} (\|\f_t\| ),
\]
   which finished the proof.
\end{proof}

We showed that the optimization gain converges to the Newton-GMRES(\(m\)) gain \(\|\r_t^J\| / \|\f_t\|\) as the residual \(\f_t\) approaches zero. 
Thus, \(\|\r_t^J\| / \|\f_t\|\) can serve as an estimator of \(\theta_t\) when the residual is small.
For a general \(\Jt\), the bounds on the \(m\)-th residual  produced by GMRES applied to \(\Jt \p = \f_t\) can vary widely\cite{liesen2004convergence}. 
Roughly speaking, this says that fast convergence occurs when the eigenvalues of A are clustered away from the origin and A is not too far from normality.
Specially, when \(\Jt\) is symmetric and positive definite, the norm of the \(m\)-th GMRES residual is well-known to be bounded as follows
\begin{equation}
\label{equ:upppedboundgain}
    \frac{\|\r_t^J\| }{ \|\f_t\|} \leq 2 \left( \frac{\sqrt{\cond(\Jt)} - 1}{\sqrt{\cond(\Jt)} + 1} \right)^m \approx 2 \left(1 - \frac{2}{\sqrt{\cond(\Jt)}} \right)^m,
\end{equation}
showing exponential decay with respect to \(m\). 
However, according to \eqref{equ:ytstbound}, when \(m\) is large, the upper bound in \eqref{equ:optimizationgainbound} could have a large constant, which makes the bound loose.
Numerical examples are provided in \Cref{fig:logistic_m-theta} in Section~\ref{sec:logistic}.

\begin{remark}
Another interpretation that illustrates the relationship between \( \theta_t \) and \( \|\r_t^J\|/\|\f_t\| \) is based on the equivalence between AAP and multisecant-GMRES.
Specifically, from \eqref{residualequation}, one can write $
\theta_t \|\f_t\| =\|\Bt\widehat{\p}_t - \f_t\|$,
for any \(\Bt \in \bbB\).
Since \((\Jt + \Et)\St = \Yt\), we have 
\[
\theta_t \|\f_t\| = \min_{\p \in \mathcal{K}_m(\Jt + \Et, \f_t)} \|(\Jt + \Et)\p - \f_t\|.
\]
Combining this with \eqref{definition:gain}, the relation between \(\theta_t\) and \(\|\r_t^J\|/\|\f_t\|\) can be regarded as the change in GMRES residuals under perturbation \(\Et\).
Interested readers can refer to~\cite{sifuentes2013gmres}, where the GMRES residual change with respect to perturbations on the coefficient matrix is estimated using spectral perturbation theory and resolvent estimates. A simple application of Theorem~2.1 therein gives
$\left|\,\theta_t - \|\r_t^J\|/\|\f_t\|\,\right| \leq \mathcal{O}(\|\Et\|)$ under some conditions.
This bound is essentially identical as \eqref{equ:optimizationgainbound}.
\end{remark}

\section{Convergence analysis}
\label{sec:conv}

This section presents the convergence analysis of AAP. \edit{
We first provide a one-step estimate for \(\|\f_{t+1}\|\) in terms of the multisecant matrix in \Cref{thm:residual_cvgce}, followed by a local $q$-convergence result for AAP in \Cref{theorem:localconv} and an upper bound of asymptotic $q$-linear convergence factor in \Cref{cor:asymptotic-qfactor}. 
We note that the one-step estimate in \Cref{thm:residual_cvgce} is different from existing ones in, e.g., \cite{evans2020proof} and \cite{pollock2021anderson}, in the sense that it does not rely on unknown constants from assumptions such as bounded least-squares solutions.
Instead, the estimate here are dependent on the properties of the underlying multisecant matrix, which follows from the equivalence between AAP and multisecant-GMRES established in Section~\ref{sec:aapequivalence}.
Furthermore, the convergence of AAP to Newton-GMRES in Section~\ref{sec:relation} allows us to quantify the asymptotic convergence behavior of AAP.}

We present some preliminary results before starting the convergence analysis.
First, recall \Cref{coro:praap} which states that the update of \( \x_{t+1} \) in AAP satisfies
\[
\x_{t+1} = \x_t - \widehat{\p}_t - \damping \widehat{\r}_t,
\]
where \(\widehat{\r}_t = \Bt \widehat{\p}_t - \f_t\) is the residual of the multisecant-GMRES. 
It follows from the optimization gain \(\theta_t = \|\Bt \widehat{\p}_t - \f_t\| / \|\f_t\|\) that \(\| \widehat{\r}_t \| = \theta_t \| \f_t \|\).
Meanwhile,
the following lemma characterizes the lengths of $\| \widehat{\p}_t \|$.

\begin{lemma}\label{lemma:boundpp}
For any invertible $\Bt \in \bbB$, we have
$\widehat{\p}_t$  satisfies \[\textstyle
\| \widehat{\p}_t \| \leq \sqrt{ 1-\theta_t^2 }\|\Bt^{-1}\| \|\f_{t}\|.
    \]
\end{lemma}

\begin{proof}
    Since $\widehat{\p}_t = \argmin_{\p \in \mathcal{K}_{\E}(\Bt, \f_t)}   \| \Bt \p - \f_t \|$, $ \Bt\widehat{\p}_t \perp (\Bt \widehat{\p}_t  - \f_t)$.
    Thus, $\| \Bt\widehat{\p}_t\| = \sqrt{ 1-\theta_t^2 }\|\f_{t}\|$ as $ \|\Bt \widehat{\p}_t  - \f_t\|  =\theta_t \|\f_{t}\|$.
    The result follows from the nonsingularity of $\Bt$.
\end{proof}

We note that assuming $\Bt$ nonsingular does not impose additional restrictions, since the equivalence of AAP and multi-secant GMRES holds for any $\Bt\in\bbB$ (Section~\ref{sec:aapequivalence}) and $\bbB$ has nonsingular elements under \Cref{assu:syfullrank} as shown in Section~\ref{sec:multi}.

The above bound holds for any invertible $\Bt\in \bbB$.
We give a lower bound estimate of  $\|\Bt^{-1}\|$: 
considering that $\Bt^{-1}$ satisfies $\Bt^{-1} \Yt = \St$,
we have,
according to the minimal norm properties of the pseudoinverse,
$\|\Bt^{-1}\| \geq \|\St \Yt^{\dagger}\|,
\forall\, \Bt\in \bbB
$ being invertible.
On the other hand,
if $\Jt$ is invertible and $\|\Jt^{-1}\Et\|<1$, \cite[Theorem~2.3.4]{golub2013matrix} gives that $\Jt+\Et \in \bbB$ is invertible and 
\begin{equation}
    \label{matrixinversebound}
     \| (\Jt+\Et)^{-1} - \Jt^{-1}\| \leq 
    \frac{\|\Et\| \, \|\Jt^{-1}\|^2 }{1- \|\Jt^{-1}\Et\|}.
\end{equation}
When $\|\Et\|$ is small enough, the denominator is negligible.
It follows that  $\| (\Jt+\Et)^{-1} - \Jt^{-1}\| \leq \mathcal{O}( \|\Et\|  )$ 
and $\|\Jt^{-1}\|$ can be used to estimate $\|\Bt^{-1}\|$ when $\|\Et\|$ is small enough.

\medskip
Now, we are ready to give the bound of $\|\f_{t+1}\|$.

\begin{theorem}\label{thm:residual_cvgce}
Under Assumptions~\ref{assumption:gaap} and \ref{assm:condst},
for any invertible $\Bt \in \bbB$,
\begin{equation}
\label{equ:boundresidual}
        \|\f_{t+1}\| \leq \left[(1-\damping) + \damping \contraction\right] \theta_t \|\f_t\| + \sqrt{1 - \theta_t^2} \|\Bt^{-1}\| \left[\frac{\gamma}{2} \sqrt{1 - \theta_t^2}\|\Bt^{-1}\|  + C_E\right] \|\f_t\|^2.
\end{equation}
\end{theorem}

\begin{proof}
The update can be expressed as
\[
\x_{t+1} = \x_t - \p_t = \x_t - (\widehat{\p}_t + \damping \widehat{\r}_t) = \widehat{\x}_t- \damping \widehat{\r}_t,
\]
where \( \widehat{\x}_t:= \x_t - \widehat{\p}_t = \sum_\ell \alpha^{\ell} \x_t^{\ell} \).
By \eqref{eqn:line_int},
\[
\begin{aligned}    
f(\x_{t+1}) &= f(\x_t) + [f(\widehat{\x}_t) - f(\x_t)] +  [f(\x_{t+1}) - f(\widehat{\x}_t)  ] \\
&\textstyle = \f_t - \int_0^1 f'(\x_t-s\widehat{\p}_t )\widehat{\p}_t \,\d r - \int_0^1 f'(\widehat{\x}_t-s\beta_t\widehat{\r}_t )\beta_t\widehat{\r}_t \,\d s\\
&=\textstyle -\widehat{\r}_t+ \Bt\widehat{\p}_t - \int_0^1 f'(\x_t-s\widehat{\p}_t )\widehat{\p}_t \,\d s - \int_0^1 f'(\widehat{\x}_t-s\beta_t\widehat{\r}_t )\beta_t\widehat{\r}_t \,\d s\\
&=\underbrace{\textstyle-\int_0^1 [\bfI+\damping f'(\widehat{\x}_t-s\beta_t\widehat{\r}_t )]\widehat{\r}_t \,\d s}_{:=\mathcal{L}_t}+ \underbrace{\textstyle\int_0^1[\Bt- f'(\x_t-s\widehat{\p}_t )]\widehat{\p}_t \,\d s}_{:=\mathcal{H}_t}  .
\end{aligned}
\]

For the term $\mathcal{L}_t$, we have $\bfI+\damping f'(\x_t-s\widehat{\r}_t ) = (1-\damping )\bfI + \damping [\bfI+f'(\x_t-s\widehat{\r}_t  )]$.
Note that $ \bfI+f'(\x_t-s\widehat{\r}_t  )= g'(\x_t-s\widehat{\r}_t  )$ and $\|g'(\x)\| \leq \contraction, \forall\,\x$.
Thus,
\begin{equation}
    \label{equ:L}
    \begin{aligned}
    \|\mathcal{L}_t\| &\leq 
\int_0^1 \| \bfI+\damping f'(\x_t-s\widehat{\r}_t  )\| \| \widehat{\r}_t\| \,\d s \leq [(1-\damping ) + \damping \contraction] \theta_t\|\f_t\|.
\end{aligned}
\end{equation}

The term $\mathcal{H}_t= \int_0^1 [\Bt - f'(\x_t-s\widehat{\p}_t) ] \widehat{\p}_t \,\d s $ is equal to $ \int_0^1 [\Jt+\Et- f'(\x_t-s\widehat{\p}_t )] \widehat{\p}_t \,\d s $ since $\Bt \widehat{\p}_t = (\Jt+\Et) \widehat{\p}_t$ for all $\Bt \in \bbB$.
By telescoping, we have
\[
\begin{aligned}    
\|\Jt+\Et - f'(\x_t-s\widehat{\p}_t )\|&\leq \|f'(\x_t-s\widehat{\p}_t )- f'(\x_t)\|+ \|f'(\x_t)-(\Jt+\Et)\|\\
&\leq s\gamma \|\widehat{\p}_t\| +\|\Et\| \\
& \leq s \gamma \sqrt{1-\theta_t^2} \| \Bt^{-1}\|   \|\f_t\| + \|\Et\|,
\end{aligned}
\]
where the second inequality holds because $g'$ is Lipschitz continuous and
the last inequality is due to the Lemma~\ref{lemma:boundpp}. 
Thus, we have
\begin{equation}
    \label{equ:H}
    \begin{aligned}
 \| \mathcal{H}_t   \| &\leq  \int_0^1 \|f'(\x_t-s\widehat{\p}_t)-(\Jt+\Et)\| \, \|\widehat{\p}_t\| \,\d s\\
 &\leq \sqrt{1- \theta_t^2}\| \Bt^{-1}\|\left[\frac{\gamma }{2} \sqrt{1-\theta_t^2} \| \Bt^{-1}\|   \|\f_t\| + \|\Et\| \right]\|\f_t\|.
\end{aligned}
\end{equation}

The claim then follows from  \eqref{equ:L}, \eqref{equ:H}, and the estimate $\|\Et\| \leq C_E \|\f_t\|$ given in \Cref{theorem:etbound} under \Cref{assm:condst}.

\hfill $\qed$
\end{proof}

The bound given in \eqref{equ:boundresidual} includes a linear term and a higher-order term, which is similar to those of classical AA given in~\cite{evans2020proof,pollock2021anderson}.
\edit{
However, the higher-order term here only involves $ \|\f_t\|^2$, whose coefficient is determined by the multisecant matrix $\Bt$.
We can see from \Cref{coro:btconver}, where the multisecant matrix $\Bt$ is shown to converge to the Jacobian matrix $\Jt$, that the coefficient of the higher order term can then be estimated by the properties of $\Jt$ when the residual is small.}
The following result is derived directly by applying \eqref{matrixinversebound} and \Cref{theorem:etbound}.
\begin{corollary}
\label{cor:boundJt}
Under Assumptions~\ref{assumption:gaap} and \ref{assm:condst},
assume that $\Jt$ is invertible and  that either $\f_t$, $\gamma$, or both is small enough.
Then, 
\begin{equation}
    \begin{aligned}
       \|\f_{t+1}\|
       \leq &\left[(1-\damping) + \damping \contraction\right] \theta_t \|\f_t\| \\
       &\textstyle
       +\gamma \sqrt{1 - \theta_t^2} \left\|\Jt^{-1}\right\|
       \big[\frac{1}{2} \sqrt{1 - \theta_t^2}\left\|\Jt^{-1}\right\| \|\f_t\|  + \frac{C_E}{\gamma}\big] \|\f_t\|^2+ o(\|\f_t\|^2 ).
\end{aligned} 
\end{equation}
\end{corollary}
\edit{
Recall from \Cref{theorem:etbound} that the constant $C_E$ is proportional to $\gamma$.
We can see that a small Lipschitz constant $\gamma$, which implies that $g$ is nearly linear, helps control the higher-order term above.
\Cref{cor:boundJt} also implies that an ill-conditioned 
$\Jt$ may lead to oscillations in the residuals.
Ensuring a sufficiently small higher-order term prompts a monotonic decay of the residual after each AA step.}

\begin{remark}
The higher-order term \(\mathcal{H}_t\) in the proof of \Cref{thm:residual_cvgce} equals
\[\textstyle
\mathcal{H}_t = \Bt \widehat{\p}_t + f(\x_t - \widehat{\p}_t) - f(\x_t) = f\left( \sum_{\ell=0}^m \alpha_t^{\ell} \x_t^{\ell} \right) - \sum_{\ell=0}^m \alpha_t^{\ell} f(\x_t^{\ell}),
\]
which might be positive, potentially resulting in \(\|\f_{t+1}\| > \|\f_t\|\). However, if \(f\) is convex and \(\alpha_t^\ell \geq 0\) for $\ell=0,\dots,m$, then \(\mathcal{H}_t \leq 0\) and the monotonic decay of \(\|\f_t\|\) is guaranteed, i.e., \(\|\f_{t+1}\| \leq \|\f_t\|\).
The condition \(\bfalpha_t \geq 0\) can be satisfied by imposing nonnegative constraints on the LS problem in \eqref{equ:constrainedLS}, as in the EDIIS method~\cite{chen2019convergence},
but the convergence might become slow.
\end{remark}

\medskip
Now, we present the local \edit{$q$-linear} convergence of AAP under the condition that $g$ is a contractive mapping, i.e., $\contraction < 1$. 
Under this condition, $g$ has a unique fixed point $\x^*$ such that $f(\x^*) = 0$.
Since $\|g'(\x^*)\| \leq \contraction<1$, it follows  that the Jacobian $f'(\x^*)$ is non-singular.
Furthermore, an important property of contractive mappings is the relation between the residual and error. Specifically, for all \(\x \in \bbR^d\), the error satisfies
\[
\|\x - \x^*\| \leq \| \x - g(\x)\| + \|g(\x) - g(\x^*)\| \leq \|f(\x)\| + \contraction \|\x - \x^*\|,
\]
which implies $\|\x - \x^*\| \leq (1 - \contraction)^{-1} \|f(\x)\|$.
Meanwhile, the residual satisfies $$
\|f(\x)\| = \| (\x  -\x^*) - [g(\x) - g(\x^*)]\| \leq (1+ \contraction)\|\x -\x^*\|.
$$

\begin{theorem}[Local \edit{$q$-}linear convergence of residual]
\label{theorem:localconv}
Assume Assumptions~\ref{assumption:gaap} and \ref{assm:condst} 
hold, with $\damping\geq\beta>0$ for some $\beta$, and the contraction constant $\contraction$ in \Cref{assumption:gaap} satisfies $\contraction<1$.
Let $\x^*$ be the fixed-point solution, i.e., $\x^*=g(\x^*)$.
\edit{For any constant $\rho$ satisfying $\sup_t[(1-\damping) + \damping \contraction]\theta_t <\rho <1$,
there exists $\x_0$  sufficiently close to $\x^*$ such that $\f_t$ converges $q$-linearly to $f(\x^*)=0$, i.e., $\|\f_{t+1}\| \leq \rho \|\f_{t}\|$. }
\end{theorem}

\begin{proof}

Recall the bound on $\|\f_{t+1}\|$ given in \eqref{equ:boundresidual}.
As $\contraction<1$, $\damping\geq\beta>0$ and $\theta_t\leq 1$, 
the coefficient in the linear term satisfies
$[(1-\damping) + \damping \contraction]\theta_t <1$.
Thus, there exists $\rho$ that satisfies $\sup_t[(1-\damping) + \damping \contraction]\theta_t <\rho <1$.
Additionally, the coefficient in the higher-order term includes $\|\Bt^{-1}\|$ where $\Bt \in \bbB$ is invertible.

We first find the conditions to bound \( \|\Bt^{-1}\| \). 
We choose a small \(\epsilon > 0\). According to the continuity property of the Jacobian inverse~\cite[Lemma~2.1]{dembo1982inexact} and the fact that \( f'(\x^*) \) is nonsingular, there exists \(\delta_1 > 0\) such that \( f'(\x_t) = \Jt \) is invertible and \( \| f'(\x_t)^{-1} - f'(\x^*)^{-1} \| \leq \epsilon \) provided that \( \|\x_t - \x^*\| \leq (1 - \contraction)^{-1}\delta_1 \), which can be satisfied by requiring  \( \|f(\x_t)\| \leq \delta_1 \).
Furthermore, according to \eqref{matrixinversebound}   and $\|\Et\|\leq C_E\|\f_t\|$ under Assumptions~\ref{assumption:gaap} and \ref{assm:condst}, if \(\Jt\) is invertible, there exists a sufficiently small \(\delta_2 > 0\) such that \((\Jt + \Et)\) is invertible and \(\|(\Jt + \Et)^{-1}\| \leq \|\Jt^{-1}\| + \epsilon\) when \(\|\f_t\| \leq \delta_2\). 
Note that \(\Jt + \Et \in \bbB\).
Therefore, if \(\|\f_t\| \leq \min\{\delta_1, \delta_2\}\), there exists \(\Bt \in \bbB\) which is invertible and satisfies \(\|\Bt^{-1}\| \leq M^*\) where
\(M^* := \|f'(\x^*)^{-1}\| + 2\epsilon\).

Next, we find conditions to guarantee the monotonic decay of the residual, i.e., $\|\f_{t+1}\| \leq  \rho \|\f_t\|$.
According to \eqref{equ:boundresidual}, 
\[\textstyle
\begin{aligned}
\|\f_{t+1}\| &\leq \left[(1-\damping) + \damping \contraction\right] \theta_t \|\f_t\| + \sqrt{1 - \theta_t^2} \|\Bt^{-1}\| \left[\frac{\gamma}{2} \sqrt{1 - \theta_t^2}\|\Bt^{-1}\|  + C_E\right] \|\f_t\|^2\\ \textstyle
&\leq \left[(1-\damping) + \damping \contraction\right]\theta_t\|\f_t\| +  \|\Bt^{-1}\| \left[\frac{\gamma}{2} \|\Bt^{-1}\|  + C_E\right] \|\f_t\|^2\\ \textstyle
& = \left[[(1-\damping) + \damping \contraction]\theta_t +  \|\Bt^{-1}\| \left(\frac{\gamma}{2} \|\Bt^{-1}\|  + C_E\right)\|\f_t\| \right] \|\f_t\|.
\end{aligned}
\]
We have that if \(\|\f_t\| \leq \min\{\delta_1, \delta_2\}\), 
\begin{equation}\label{eq:localqpf}
 \textstyle
\|\f_{t+1}\| \leq \left[[(1-\damping) + \damping \contraction]\theta_t +  M^* \left(\frac{\gamma}{2} M^*  + C_E\right)\|\f_t\| \right] \|\f_t\|.
\end{equation}
It follows that $\|\f_{t+1}\| \leq  \rho \|\f_t\|$ when 
$\|\f_t\| \leq \delta_3$ where 
$
\delta_3= \frac{ \rho -\sup_t[ (1-\damping) + \damping \contraction)]\theta_t} { M^* \left(\frac{\gamma}{2} M^*  + C_E\right)}>0$.
Thus, if $\|\f_t\| \leq \min\{\delta_1, \delta_2, \delta_3\}$, we have $\|\f_{t+1}\| \leq  \rho \|\f_t\|$.

Since \( \rho < 1 \), if \(\|\f_0\| \leq \min\{\delta_1, \delta_2, \delta_3\}\), then \(\|\f_{t+1}\| \leq \rho \|\f_t\|\) for all \( t \geq 0 \). This condition, \(\|\f_0\| \leq \min\{\delta_1, \delta_2, \delta_3\}\), is satisfied when \(\x_0\) is sufficiently close to \(\x^*\) due to the continuity of \( f \). This concludes the proof.

\hfill $\qed$ $\proofbox$

\end{proof}

We have established the linear convergence of the residual, \(\|\f_{t+1}\| \leq \rho \|\f_t\|\). The convergence of the iterates \(\{\x_t\}\) follows as
\[
\|\x_t - \x^*\| \leq (1 + \contraction)\|\f_t\| \leq (1 + \contraction)\rho^t\|\f_0\|.
\]

\edit{
Finally, the following \Cref{cor:asymptotic-qfactor} provides an upper bound for the asymptotic $q$-linear convergence factor of AAP.
\begin{corollary}
\label{cor:asymptotic-qfactor}
Assume Assumptions~\ref{assumption:gaap}, \ref{assm:condst}, and \ref{assum:ytgt} 
hold, with $\damping\geq\beta>0$ for some $\beta$, and the contraction constant $\contraction$ in \Cref{assumption:gaap} satisfies $\contraction<1$.
Let $\x^*$ be the fixed-point solution, i.e., $\x^*=g(\x^*)$.
When $\|\f_t\|$ converging to $0$,
the asymptotic $q$-linear convergence factor of AAP satisfies
\begin{equation}
\label{equ:uppperqfactorNonSPD}
  \lim_{t \to \infty}  \frac{\|\f_{t+1}\|}{\|\f_t\|}
  \;\le\;
   [(1-\beta) + \beta \contraction]   \max_{\|v\|=1}
 \inf _{p \in \mathbb{P}_m}\|p(f'(\x^*)) v\|,
\end{equation}
where $\mathbb{P}_m$ denotes the set of polynomials of degree at most $m$ with $p(0)=1$.
Furthermore,  when $f'(\x^*)$ is symmetric and positive definite (SPD), then
\begin{equation}
\label{equ:uppperqfactorSPD}
  \lim_{t \to \infty}  \frac{\|\f_{t+1}\|}{\|\f_t\|}
  \;\le\;
  2 [(1-\beta) + \beta \contraction] \left( \frac{\sqrt{\cond(f'(\x^*))} - 1}{\sqrt{\cond(f'(\x^*))} + 1} \right)^m.
\end{equation}
\end{corollary}

This corollary is a direct consequence from the facts that \(\|\f_{t+1}\| /\|\f_t\| \rightarrow [(1-\damping) + \damping \contraction]\theta_t \) as $\|\f_t\|\to 0$ (Eq.~\eqref{eq:localqpf}), $\theta_t $ converging to the Newton-GMRES gain (\Cref{theorem:gain}), and existing bounds on the GMRES residual, such as the one given in \eqref{equ:upppedboundgain}.
Here,
\eqref{equ:uppperqfactorNonSPD} corresponds to the \textit{worse-case} GMRES residual bound given in \cite{liesen2004convergence}, whereas \eqref{equ:uppperqfactorSPD} uses \eqref{equ:upppedboundgain} under the additional assumption that the Jacobian is SPD. 
When $f'(\x^*)$ is SPD and well-conditioned, the bound on $q$-linear convergence factor given in \eqref{equ:uppperqfactorSPD} could be much sharper than the \(\kappa^m\) factor from an \(m\)-step Picard iteration.
When $f'(\x^*)$ satisfies additional assumptions, such as normality, better estimates on the $q$-linear convergence factor could be derived from sharper GMRES residual bounds. We refer the readers to~\cite{liesen2004convergence,embree2022descriptive} for details.

}

\begin{remark}

AAP can be regarded as an inexact Newton method~\cite{dembo1982inexact} whose search direction satisfies $\| \Jt \p_t - \f_t \| \leq \eta_t \|\f_t\|$ where $\p_t := \x_t - \x_{t+1}$.
This \( \eta_t \) is called the forcing term of the inexact Newton method and assesses how well \( \p_t \) approximates the exact Newton direction in each iteration. For any invertible \(\Bt \in \bbB\), we have
\[
\|\Jt\p_t - \f_t\| \leq \left[ [(1-\damping) + \damping \contraction]\theta_t + \sqrt{ 1 - \theta_t^2 }\|\Et\| \|\Bt^{-1}\| \right] \|\f_t\|.
\]
Since \(\p_t = \widehat{\p}_t + \damping \widehat{\r}_t\), the above bound comes from the decomposition of \(\Jt \p_t - \f_t = (I + \damping \Jt) \widehat{\r}_t - \Et \widehat{\p}_t\). Thus, the forcing term of AAP satisfies \(\eta_t = [(1-\damping) + \damping \contraction]\theta_t + \sqrt{1 - \theta_t^2}\|\Et\| \|\Bt^{-1}\|\).
If \(\|\Et\| \leq C_E \|\f_t\|\), then \(\eta_t \leq [(1-\damping) + \damping \contraction]\theta_t + \mathcal{O} (\|\f_t\|)\). It follows that, when \(\|\f_t\|\) approaches 0, \(\eta_t\) converges to \([(1-\damping) + \damping \contraction]\theta_t\), with the optimization gain \(\theta_t\) converging to the Newton-GMRES gain as shown in \Cref{theorem:gain}.
One sufficient condition to guarantee the local convergence of AAP, as an inexact Newton method, is that \(\sup_t \eta_t < 1\)~\cite{dembo1982inexact}, which can be ensured by the same assumptions as in \Cref{theorem:localconv}.
    
\end{remark}

\section{Numerical results}
\label{sec:numerical-results}

In this section, we present numerical experiments to demonstrate the performance of the AAP($m$) against other algorithms, including:

\begin{itemize}
    \item Picard: Picard  iteration with $\x_{t+1} = g(\x_t)$.
    \item AA($m$)~\cite{walker2011anderson}: Classical Anderson acceleration with a fixed window size $m$, taking an AA step at each iteration.
    This way, the algorithm is restarted every $m+1$ iteration.
    \item resAA($m$)~\cite{ouyang2020anderson}: Restarted Anderson acceleration, where an AA step is taken at each iteration. The number of history iterates used increases until a given threshold $m$ is reached, after which the number of history iterates used is reset to 0.
    \item Newton-GMRES\cite{choquet1993some}: Solves \(\Jt \p_t^{\textup{N}} = \f_t\) using GMRES($m$) at each global iteration, updating \(\x_{t+1} = \x_t - \p_t^{\textup{N}}\). There is no line search applied.
\end{itemize}
Note that AAP($m$) evaluates $m+1$ Picard iterations in each global iteration, whereas Picard, AA($m$), and resAA($m$) evaluate one Picard iteration per iteration. 
In all experiments, we used the damping ratio $\beta_t=1$.
The performance of each algorithm may vary depending on the specific problem and the parameters. The examples provided
are chosen to better
illustrate the features of AAP.
The code can be accessed at \url{https://github.com/xue1993/AAP.git}.

\subsection{Logistic regression}
\label{sec:logistic}

The first example is a nonlinear fixed-point problem derived from the gradient descent (GD) algorithm to minimize a function \( h(\x) \). The GD step with stepsize $\eta$ is formulated as  $
\x_{t+1} = \x_t - \eta \nabla h(\x_t)$,
which can be regarded as the Picard iteration on the fixed-point problem:
\[
\x = g(\x) = \x - \eta \nabla h(\x).
\]
Here, \( h \) is the loss function of a regularized logistic regression problem. Specifically,
\[
h(\x) = \frac{1}{n} \sum_{i=1}^{n} \log(1 + \exp(-y_i \x^\top \v_i)) + \frac{\mu}{2} \|\x \|^2,
\]
where $\v_i \in \mathbb{R}^d$ is a feature vector and  $y_j \in\{-1,1\}$ is the corresponding response, \( \x \) represents the weights, and \( \mu \) is the regularization parameter.
We use two classical datasets: \texttt{w8a} $(n=49,749$ and $d=300)$ and \texttt{covtype} ( $n=581,000$ and $d=54)$. Both datasets are available at the LIBSVM website~\cite{Chang2011LIBSVM}.
The function
$h$ is strongly convex.
For all examples, we use \(\eta = 1\) and ensure that \(g\) is a contractive mapping. The global minimizer is denoted as \(\x^*\), and unless specified, the default initial guess is \(\x_0 = 0\).

\Cref{fig:logistic_compare} shows the results of different algorithms on the \texttt{w8a}  and \texttt{covtype} datasets.
From \Cref{fig:logistic_compare}(a) about \texttt{w8a} dataset,
we observe that AAP outperforms other algorithms.
The convergence plot for AAP in terms of the number of Picard iterations exhibits a stair-step effect, with flat regions in each global iteration from Picard steps  and rapid decay from the AA step marked as red dot. Meanwhile, the errors of resAA and AA oscillate, especially in the first few iterations. This oscillation is also observed in the landscape of function $h$ and optimization path plots in  \Cref{fig:logistic_compare}(c), where AA and resAA overshoot around the global minimizer, slowing their convergence.
  This phenomenon can be explained as their approximated Jacobians at the AA step are inaccurate as they use points that are far from each other.

For the \texttt{covtype} dataset, both AAP and resAA perform well, as shown in \Cref{fig:logistic_compare}(b).
The optimization paths of all AA algorithms in \Cref{fig:logistic_compare}(d) also show less oscillation on this dataset.
The third row in \Cref{fig:logistic_compare} illustrates the impact of varying \( m \) on different AA algorithms. It is observed that increasing \( m \) does not necessarily improve the convergence rate for all AA variants.
Compared to AA and resAA, the performance of AAP demonstrates less oscillation across different values of \( m \). 
Nonetheless, the convergence rate of  AAP  tends to decrease over iterations.
This is likely due to machine round-off error, as the points generated by Picard are very close to each other, making it difficult to retain 
useful
curvature information from $\Yt$.

\begin{figure}[htbp]
    \centering
    \subfloat[\texttt{w8a}($m=5$)]{\includegraphics[width=0.47\textwidth]{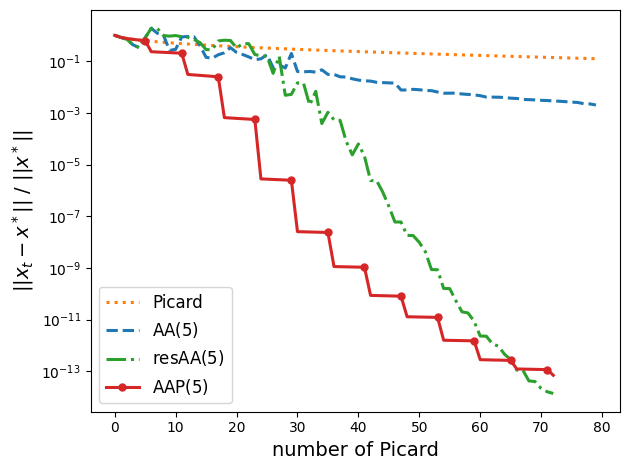}} 
    \hfill
    \subfloat[\texttt{covtype}($m=5$)]{\includegraphics[width=0.47\textwidth]{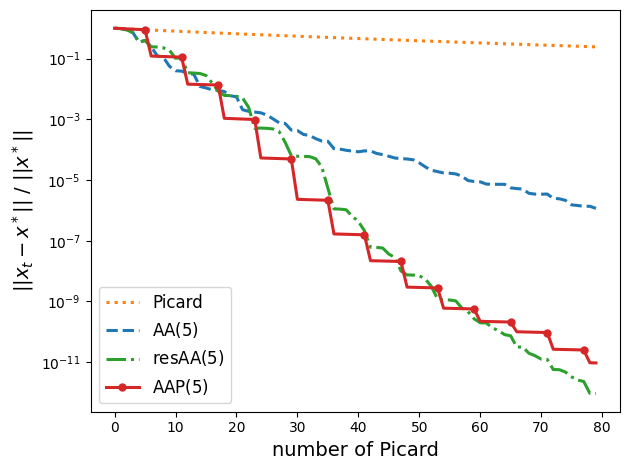}}

    \subfloat[\texttt{w8a}: optimization paths($m=5$)]{\includegraphics[width=0.47\textwidth]{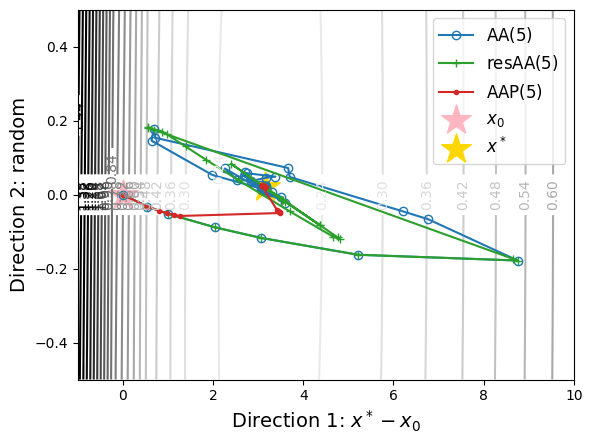}}        
    \hfill
    \subfloat[\texttt{covtype}: optimization paths($m=5$)]{\includegraphics[width=0.47\textwidth]{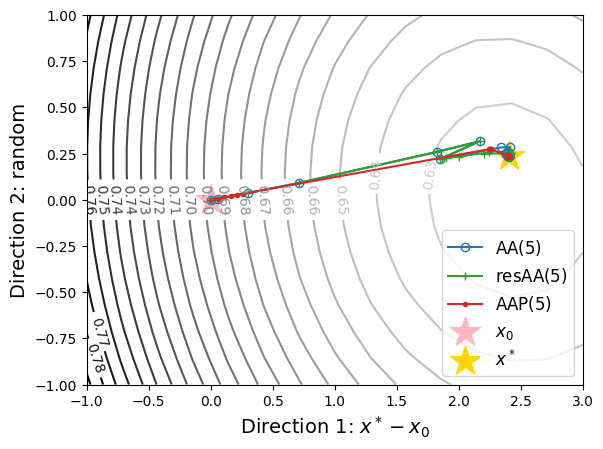}} 

    \subfloat[\texttt{w8a}: varying $m$ ]{\includegraphics[width=0.47\textwidth]{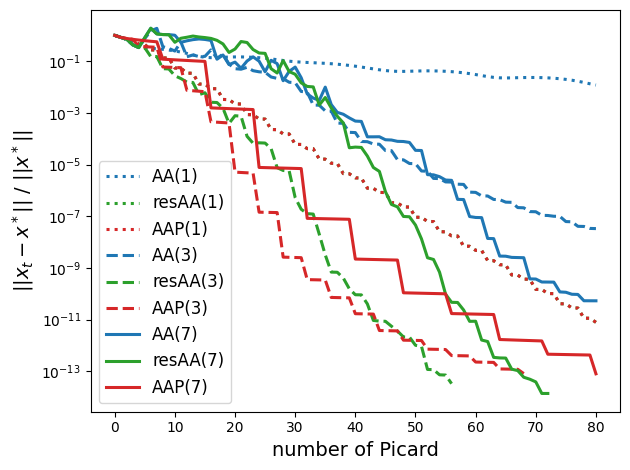}}        
    \hfill
    \subfloat[\texttt{covtype}: varying $m$]{\includegraphics[width=0.47\textwidth]{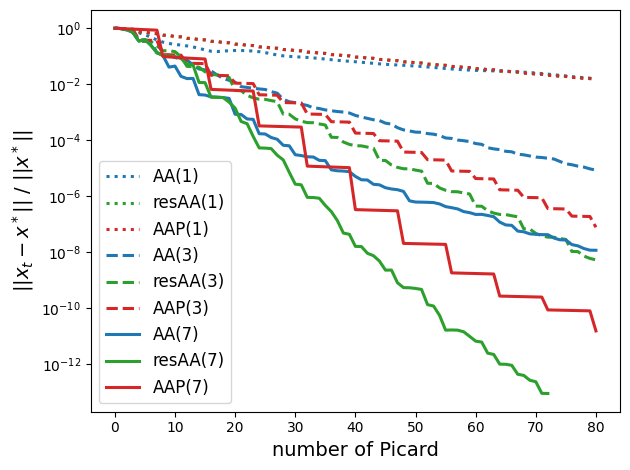}}

    \caption{Comparison of different algorithms applied to logistic regression with \(\mu = 0.01\) on both \texttt{w8a} and \texttt{covtype} datasets. First row: number of Picard iterations versus relative error. Second column: Landscape of the function $h$ and the optimization paths of different algorithms. The true domain dimension of $h$ is $300$ for the \texttt{w8a} dataset and $54$ for the \texttt{covtype} dataset. 
    To visualize $h$, we show the level set of \(h\) along two directions starting from the initial point \(\x_0 = 0\): normalized \(\x_0 - \x^*\) and a random direction. 
    The optimization paths are also projected onto these two directions,
    with each marker representing one Picard iteration.  Due to the inefficiency of the history data points, the first $m$ iterations of AA($m$) are equivalent to those of resAA($m$), with their plots overlapping. Third row: effect of varying $m$ on different AA algorithms. It is noteworthy that when \(m = 1\), AAP($1$) is equivalent to resAA($1$), resulting in the vanishing plot of resAA($1$).}
    \label{fig:logistic_compare}
\end{figure}

\edit{
The first row of \Cref{fig:logistic_m-theta} reports the residual norm $\| f(\x_t)\|$  versus the global iteration count \(t\) when solving the logistic regression problem on the \texttt{w8a} dataset using the Picard iteration, Newton-GMRES, and AAP with varying $m$. 
Here the regularization parameter is chosen to be \(\mu = 0.0001\), which leads to a contraction constant $\contraction \approx 0.9999$, and thus the convergence of the Picard iteration is slow.
In all three cases, significant improvements over the Picard iteration are observed for both AAP and Newton-GMRES.
The similar convergence behavior between AAP and Newton-GMRES confirms the connections established in Section~\ref{sec:relation}.
Furthermore, it is evident that the residual $\| f(\x_t)\|$ of AAP decreases monotonically when the residual is sufficiently small. 
This observation agrees with the local $q$-linear convergence result of AAP established in \Cref{theorem:localconv}.

The second row of \Cref{fig:logistic_m-theta} shows the optimization gain $\theta_t$ of AAP, the associated Newton-GMRES gain, and the upper bound on the Newton-GMRES gain given in \eqref{equ:upppedboundgain} in terms of the global iteration count \(t\). 
Figures~\ref{fig:logistic_m-theta}(d) and \ref{fig:logistic_m-theta}(e) demonstrate that the AAP optimization gain $\theta_t$ converges to the Newton-GMRES gain, as proved in \Cref{theorem:gain}. 
However, in the case of \(m = 7 \) considered in Figure~\ref{fig:logistic_m-theta}(f), $\theta_t$ deviates from the Newton-GMRES gain when the residual is small.
We suspect that this discrepancy is caused by inaccurate least-squares solutions in AAP due to ill-conditioning.
For each value of $m$, we observe from the pair of figures, e.g., Figures~\ref{fig:logistic_m-theta}(a) and \ref{fig:logistic_m-theta}(d), that the convergence rate of AAP is largely influenced by $\theta_t$ and that the convergence is faster with smaller $\theta_t$, as expected from the theoretical estimates in Section~\ref{sec:conv}.
}

\begin{figure}[htbp]
    \centering
    \subfloat[$m=1$]{\includegraphics[width=0.33\textwidth]{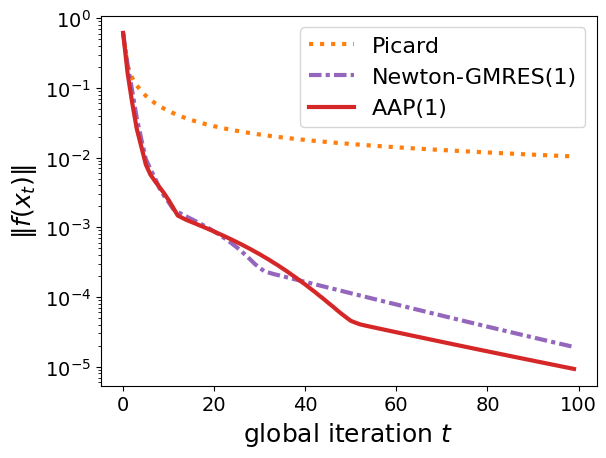}} 
    \subfloat[$m=3$]{\includegraphics[width=0.33\textwidth]{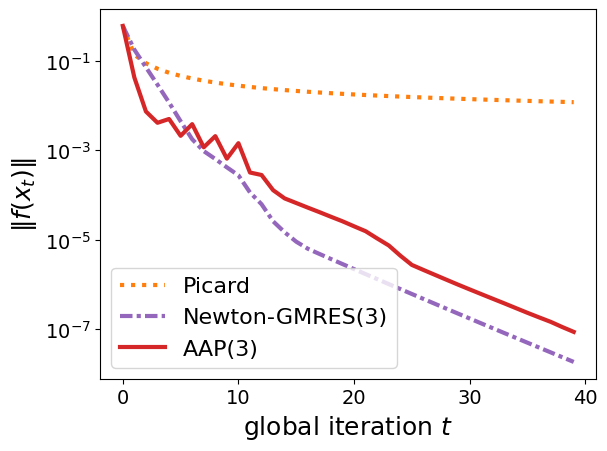}} 
    \subfloat[$m=7$]{\includegraphics[width=0.33\textwidth]{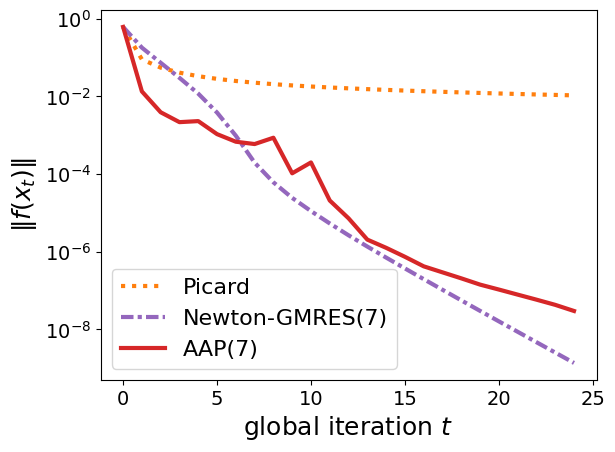}} \\
\hspace{0.15cm}
    \subfloat[$m=1$]{\includegraphics[width=0.31\textwidth]{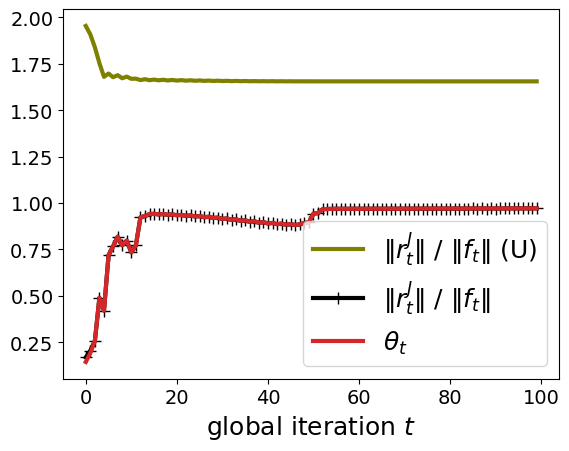}} 
    \hspace{0.05cm} \hfill
    \subfloat[$m=3$]{\includegraphics[width=0.31\textwidth]{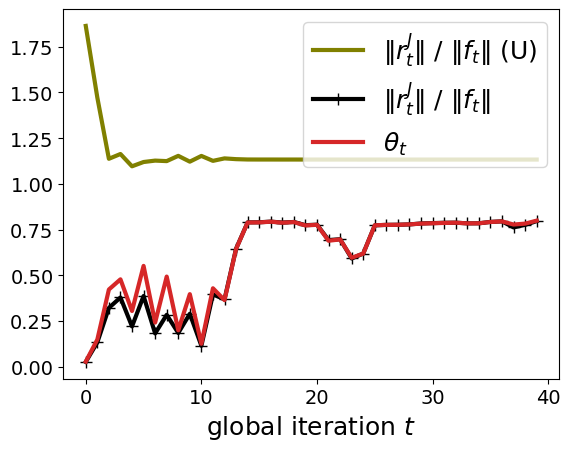}} 
    \hspace{0.05cm}\hfill
    \subfloat[$m=7$]{\includegraphics[width=0.31\textwidth]{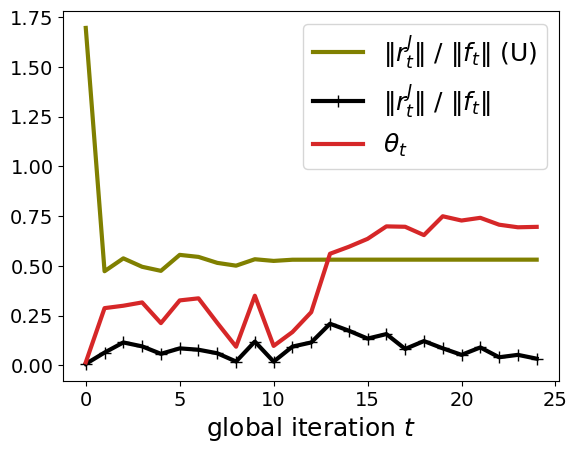}} 
    
    \caption{Convergence of AAP on logistic regression with \(\mu = 0.0001\) on the \texttt{w8a} dataset. 
    The first row shows global iterations versus residual norm $\|f(\x_t)\|$ for different algorithms, including Picard and Newton-GMRES as reference.
    In each global iteration, Picard performs \(m+1\) steps to have a fair comparison with AAP. 
    The second row presents global iterations versus optimization gain \(\theta_t\). 
    We also display the value of the Newton-GMRES gain \(\|\r^J_t\|/\|\f_t\|\), as defined in \eqref{definition:gain}. 
    The Newton-GMRES gain \(\|\r^J_t\|/\|\f_t\|\) (U) represents the theoretical upper bound of \(\|\r^J_t\|/\|\f_t\|\) given in \eqref{equ:upppedboundgain}. Here the Jacobian is the Hessian of \(h\), which is symmetric and positive definite.}
    \label{fig:logistic_m-theta}
\end{figure}

\edit{
Finally, we investigate in Figure~\ref{fig:qfactor} the behavior of the convergence factor \(\frac{\|\f_{t+1}\|}{\|\f_{t}\|}\) of AAP under various settings of \(m\) and \(\mu\). 
For all three examples, the factor eventually drops below one, indicating that the iterations have entered a \(q\)-linear convergence regime but with some oscillation in the $q$-linear convergence factor. The red lines represent the theoretical upper bound of the asymptotic $q$-linear convergence factor given in 
\eqref{equ:uppperqfactorSPD}, which closely matches the maximal observed factors shown in Figure~\ref{fig:qfactor} for 
different random initial guesses when $\mu=0.01$. 
When \(\mu = 0.0001\), the observed factors often exceed one 
during the first 15 iterations.
Since \(\|\Jt^{-1}\| \approx 1/\mu\), this phenomenon is consistent with Corollary~\ref{cor:boundJt} where we show a large \(\|\Jt^{-1}\|\) may lead to oscillations in \(\|\f_{t}\|\) before it 
eventually converges. Overall, these observations validate the theoretical estimates for the asymptotic convergence behavior of AAP in Section~\ref{sec:conv}.
}

\begin{figure}[htbp]
    \centering
    \subfloat[$\mu=0.01, m=1$]{\includegraphics[width=0.33\textwidth]{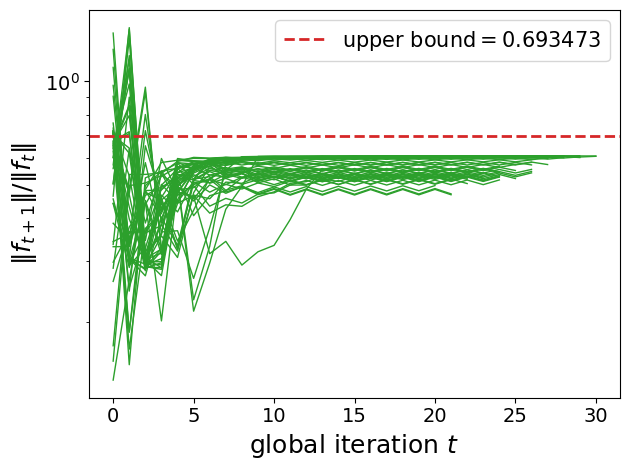}} 
    \subfloat[$\mu=0.01, m=3$]{\includegraphics[width=0.33\textwidth]{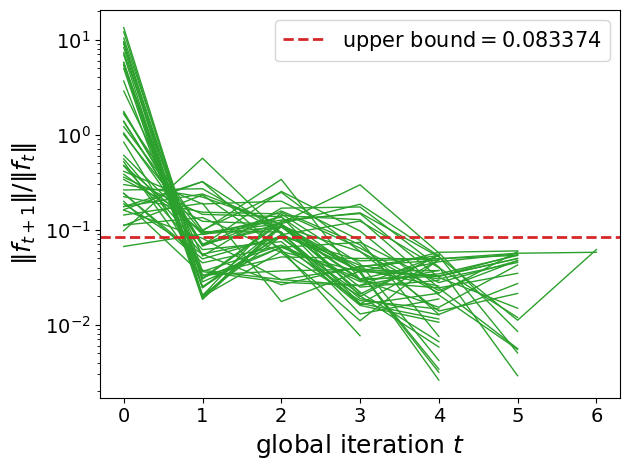}} 
    \subfloat[$\mu=0.0001,m=3$]{\includegraphics[width=0.33\textwidth]{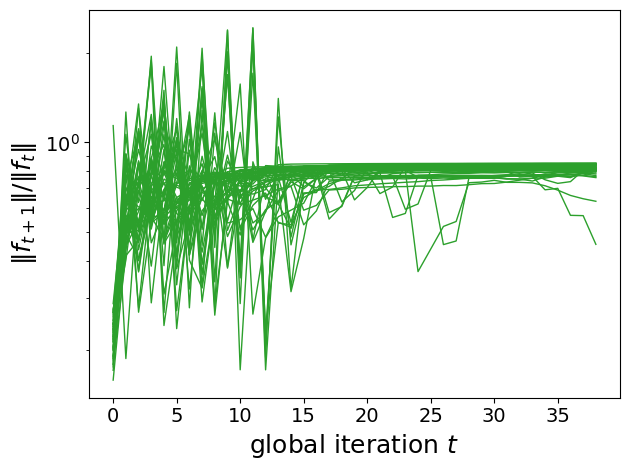}} 
    
    \caption{
 \edit{   
 Residual convergence factor \(\|\f_{t+1}\| / \|\f_{t}\|\) of AAP\((m)\) on the logistic regression problem with the \texttt{w8a} dataset under various \(m\) and \(\mu\).
 The red lines represent the theoretical upper bound given in \eqref{equ:uppperqfactorSPD}. The green lines show the observed factors from different randomly sampled initial guesses.  
        The iteration terminates once the residual norm reaches $10^{-8}$. 
        When \(\mu = 0.0001\), the condition number of \(f'(\x^*)\) is approximately 112, leading to an upper bound nearly $1$. Therefore, we do not plot the upper bound for this example.
    }}
    \label{fig:qfactor}
\end{figure}

\subsection{Nonnegative matrix factorization}
The second example is the Nonnegative Matrix Factorization (NMF) problem, which has been applied in fields such as text mining, image processing, and bioinformatics~\cite{elden2019matrix}. NMF decomposes a matrix \( \bfA \in \mathbb{R}^{d_1 \times d_2} \) into two nonnegative matrices \( \bfW \in \mathbb{R}^{d_1 \times r} \) and \( \bfH \in \mathbb{R}^{r \times d_2} \), i.e., \( \bfA \approx \bfW\bfH \) and $\bfW\geq 0, \bfH\geq 0$. The optimization problem associated with NMF can be formulated as
\[
\min_{\bfW, \bfH} \|\bfA - \bfW\bfH\|_F^2 \quad \text{subject to} \quad \bfW \geq 0, \; \bfH \geq 0,
\]
where \( \|\cdot\|_F \) denotes the Frobenius norm.  A widely used method 
to solve this problem
is the alternating nonnegative least-squares (ANNLS) method: starting from \( \bfW_0 \geq 0 \), one generates a sequence of \( \left(\bfW_t, \bfH_t\right) \) by alternately solving the standard nonnegatively constrained linear least-squares problems:
\[
\bfH_t = \argmin_{\bfH \geq 0} \|\bfA - \bfW_{t-1} \bfH\|_F \quad \text{and} \quad \bfW_{t} = \argmin_{\bfW \geq 0} \|\bfA - \bfW \bfH_t\|_F.
\]
Additionally, a normalization of \( \bfW_t \) is applied at the beginning of each iteration. Following~\cite{walker2011anderson}, we consider ANNLS to be a fixed-point iteration through the assignment \( \left(\bfW_t, \bfH_t\right) \rightarrow \left(\bfW_{t+1}, \bfH_{t+1}\right) \). We perform the AA step on this problem by solving a constrained LS based on the vectorized variable, and we truncate \( \left(\bfW_t, \bfH_t\right) \) after the AA step to guarantee nonnegativity of the sequence.

We construct synthetic NMF problems by generating random matrices \( \bfW \in \mathbb{R}^{300 \times r} \) and \( \bfH \in \mathbb{R}^{r \times 50} \), and then take \( \bfA =\bfW\bfH\in \mathbb{R}^{300 \times 50} \). Consequently, the minimum of the NMF objective function is zero. The initial point \( \bfW_0 \) is chosen randomly. \Cref{fig:NMF} presents the results of different algorithms with the same initialization. It is evident that all AA variants, including AAP, AA, and resAA, significantly improve convergence compared to ANNLS. Since the performance of each algorithm can be influenced by different initializations, we also provide the median and interquartile range plots of 15 runs.

\begin{figure}[htbp]
    \centering
    \subfloat[]{\includegraphics[width=0.47\textwidth]{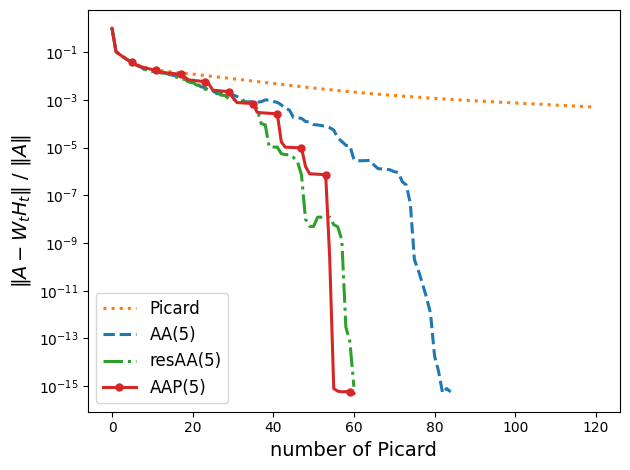}} 
    \hfill
    \subfloat[ ]{\includegraphics[width=0.47\textwidth]{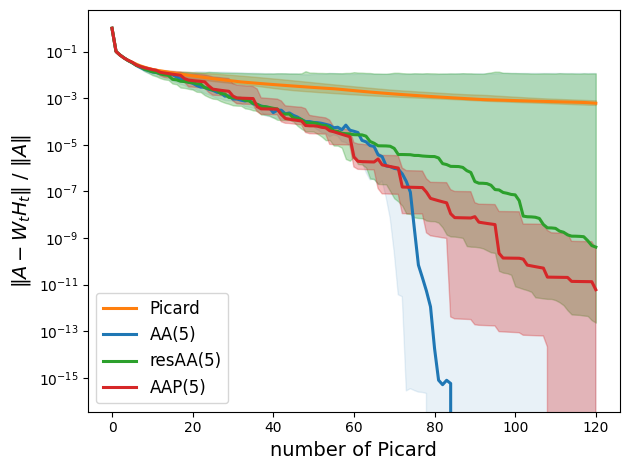}}

    \caption{Comparison of different algorithms on NMF problems with $r=4$.  All plot shows the number of Picard iterations versus the relative residual $\|\bfA - \bfW_t\Ht\|_F/\|\bfA\|_F$. (a): Different algorithms with the same initialization. (b): Median and interquartile range plots of 15 runs with random initializations. 
    }
    \label{fig:NMF}
\end{figure}

\section{Conclusion}

This work focuses on the analysis of the Alternating Anderson Picard method, in which an Anderson acceleration step is periodically applied after a number of Picard iterations. We established the equivalence between AAP and multisecant-GMRES and explored the relation between AAP and Newton-GMRES, with particular emphasis on convergence analysis. 
We proved that the AAP residual converges locally at an improved rate when compared to Picard iteration. 
Numerically, this method has demonstrated efficiency and robustness in our evaluations and is comparable to the Newton-GMRES method without accessing the Jacobian. This underscores the potential of AAP as a practical alternative in scenarios where computational resources are limited or when explicit Jacobian computation is impractical.
\edit{
Potential directions for future investigations include the extension of theoretical results analyzed in this paper to other variants of AA, as well as the development of more general alternating Anderson-Picard methods in which the AA step may involve histories from previous AA steps with potential combination with the restarting strategies considered in \cite{ouyang2024descent,fang2009two}.
}

\section*{Acknowledgments}
X.F.\ would like to express her sincere gratitude to Professor Zhaojun Bai for his valuable suggestions and insightful feedback, which contributed to the improvement of this work. 

\appendix

\section{Uniform boundedness of $\cond (\St)$}
\label{sec:boundsing}
The uniform boundedness of \(\cond(\St)\) is important for the bound of $\Et$ as shown in \Cref{theorem:etbound}.
Thus, in this  section,
we discuss $\cond(\St)$ when the AAP residual $\f_t$ approaches $ 0$.

We have shown that \(\St\) converges to \(\Gt\) as \(\|\f_t\|\) approaches zero. Consequently, \(\cond(\Gt)\) can serve as an estimator for \(\cond(\St)\) when \(\|\f_t\|\) is sufficiently small. The matrix \(\Gt\) is a Krylov matrix. The condition number and the least singular value of Krylov matrices are active research areas~\cite{beckermann2000condition,dax2017numerical,bazan2000conditioning}. Generally speaking, \(\Gt\) can be ill-conditioned even for a moderate number of columns~\cite{beckermann2000condition}.

In the following, we discuss \(\cond(\Gt)\) in a simple case where \(g'(\x_t)\) is symmetric, allowing \(\Gt\) to be decomposed into the product of a diagonal matrix and a Vandermonde matrix.

\begin{lemma}
\label{lemma:vandermonde}
Let \(g'(\x_t)\)be a symmetric matrix with eigendecomposition \(g'(\x_t) = \bfQ^T \bfLambda \bfQ\), where \(\bfQ\) is an orthogonal matrix and \(\bfLambda = \diag(\lambda_1, \ldots, \lambda_d)\).
Denote \(\bfQ \frac{\f_t}{\|\f_t\|} = [a_1, a_2, \ldots, a_d]^T\).
If \(\min_i |a_i| > 0\),
then
\begin{equation}
    \cond(\Gt) \leq \frac{\sqrt{m}}{\left(\min_i |a_i|\right)\,\sigma_{\min}\left(\bfV_m(\lambda_1, \lambda_2, \ldots, \lambda_d)\right)},
\end{equation}
where \(\bfV_m(\lambda_1, \ldots, \lambda_d)\) is a \(d \times m\) rectangular Vandermonde matrix.
\end{lemma}

\begin{proof}
Denote \(\bfK = g'(\x_t)\).
Since \(\bfK = \bfQ^T \bfLambda \bfQ\), we have \(\bfK^{(\ell)} = \bfQ^T \bfLambda^{(\ell)} \bfQ\). By the definition of \(\Gt\), we have
\[
\begin{aligned}
    \bfQ\Gt &= \bfQ[\f_t, \bfK^{(1)} \f_t, \bfK^{(2)} \f_t, \ldots, \bfK^{(m-1)} \f_t] \\
    &= [\bfQ\f_t, \bfQ\bfK^{(1)} \f_t, \bfQ\bfK^{(2)} \f_t, \ldots, \bfQ\bfK^{(m-1)} \f_t] \\
    &= [\bfQ\f_t, \bfLambda^{(1)} \bfQ\f_t, \bfLambda^{(2)} \bfQ\f_t, \ldots, \bfLambda^{(m-1)} \bfQ\f_t].
\end{aligned}
\]
It follows from \(\bfQ \frac{\f_t}{\|\f_t\|} = [a_1, a_2, \ldots, a_d]^T\) that
\[
\begin{aligned}
    \frac{1}{\|\f_t\|} \bfQ\Gt& = \begin{bmatrix}
a_1 & a_1 \lambda_1 & a_1 \lambda_1^{2} & \cdots & a_1 \lambda_1^{m-1} \\
a_2 & a_2 \lambda_2 & a_2 \lambda_2^{2} & \cdots & a_2 \lambda_2^{m-1} \\
\vdots & \vdots & \vdots & \ddots & \vdots \\
a_d & a_d \lambda_d & a_d \lambda_d^{2} & \cdots & a_d \lambda_d^{m-1}
\end{bmatrix} \\
&= \diag(a_1, a_2, \ldots, a_d) \bfV_m(\lambda_1, \lambda_2, \ldots, \lambda_d).
\end{aligned}
\]
Denote $\bfD = \diag(a_1, a_2, \ldots, a_d)$ and $ \bfV= \bfV_m(\lambda_1, \lambda_2, \ldots, \lambda_d)$.
Since \(\bfQ\) is an orthogonal matrix, we have
\[\textstyle
\sigma_{\min}\left(\frac{1}{\|\f_t\|} \Gt\right) = \sigma_{\min}\left(\frac{1}{\|\f_t\|} \bfQ\Gt\right)
= \sigma_{\min}\left(\bfD\bfV\right).
\]

Assume $\sigma_{\min}(\bfD\bfV) = \| \bfD\bfV\bar{\x }\|$ with $\|\bar{\x }\|$=1.
Since $\sigma_{\min}( \bfV) =  \min _{\x  \neq 0} \frac{\| \bfV \x \|}{\|\x \|}$,
\[ \textstyle
\sigma_{\min}( \bfV)
\leq  \frac{\| \bfV\bar{\x }\|}{\|\bar{\x }\|}  =  
\frac{\|\bfD^{-1} \bfD \bfV\bar{\x }\|}{\|\bar{\x }\|}
\leq \|\bfD^{-1}\| \,\| \bfD \bfV\bar{\x }\|
= \frac{1}{\min_i |a_i|} \| \bfD \bfV\bar{\x }\|.
\]
Thus, $\sigma_{\min}\left(\frac{1}{\|\f_t\|} \Gt\right) \geq \left(\min_i |a_i|\right)\,\sigma_{\min}\left(\bfV\right)$.
On the other hand,
$\sigma_{\max}\left(\frac{1}{\|\f_t\|} \Gt\right) = \|\left(\frac{1}{\|\f_t\|} \Gt\right)\|
\leq \|\left(\frac{1}{\|\f_t\|} \Gt\right)\|_F\leq \sqrt{m}$. 
Combining these with \[ \textstyle
\cond(\Gt) =
\cond( \frac{1}{\|\f_t\|}\Gt)
= \sigma_{\max}\left(\frac{1}{\|\f_t\|} \Gt\right) / \sigma_{\min}\left(\frac{1}{\|\f_t\|} \Gt\right),
\]
gives the result.$\qed$
\end{proof}

\vspace{0.3cm}
In \cref{lemma:vandermonde},
the condition that $\min_i |a_i| > 0$ is equivalent to the vector \(\f_t\) not being parallel to any of the eigenvectors of the matrix \(g'(\x_t)\).
The least singular value of the Vandermonde matrix $\bfV_m(\lambda_1, \ldots, \lambda_d)$ depends on the distribution of its seeds \(\lambda_i\)~\cite{bazan2000conditioning, kunis2021condition, dax2017numerical}.
Theorem 6 of \cite{dax2017numerical} shows that when the largest seed \(|\lambda_1| < 1\), the least singular value \[\sigma_{\min}(\bfV_m(\lambda_1, \ldots, \lambda_d)) \leq \sigma_{\max}(\bfV_m(\lambda_1, \ldots, \lambda_d)) |\lambda_1|^{(m)},\] which decays exponentially with respect to \(m\). 
It indicates that a large $m$ may lead to a significant increase in \(\cond( \Gt)\),
which in turn can cause a dramatic increase in $\|\Et\|$.
In addtion, since $\St   \rightarrow  \Gt$, and $ \Yt \rightarrow  \Gt^{\dagger}$,
large $m$ may cause instability of the least square problem in the AA step.
However, when \(\f_t\) is small, machine round-off errors will also affect the numerical value of  \(\cond( \St)\) and \(\cond( \Yt)\).

\bibliographystyle{siamplain}
\bibliography{refs}
\end{document}